\documentclass[preprint, 12pt]{elsarticle}
\usepackage{color,amsthm,amsmath,amsfonts,amssymb,mathtools,bm,framed,fixltx2e,hyperref}
\usepackage{bbm}
\usepackage{lineno,hyperref}
\usepackage{mathrsfs}
\usepackage{amscd}
\usepackage{enumerate}

-\marginparwidth \oddsidemargin -9mm \evensidemargin \oddsidemargin
\advance\hoffset by 5mm

\newtheorem{thm}{Theorem}
\newtheorem*{thm*}{Theorem}
\newtheorem{cor}{Corollary}
\newtheorem{lem}{Lemma}
\newtheorem{prop}{Proposition}
\theoremstyle{definition}
\newtheorem{defn}{Definition}
\theoremstyle{remark}
\newtheorem{rem}{Remark}
\theoremstyle{Hypothesis}

\DeclareMathOperator{\loc}{loc}
\DeclareMathOperator{\supp}{supp}

\numberwithin{equation}{section} \numberwithin{lem}{section}
\numberwithin{thm}{section} \numberwithin{prop}{section}
\numberwithin{cor}{section} \numberwithin{rem}{section}\numberwithin{hyp}{section}
\allowdisplaybreaks

\modulolinenumbers[5]

\begin{document}

\begin{frontmatter}

\title{Cauchy problems for Keller-Segel type time-space fractional diffusion equation}

\author[mysecondaryaddress]{Lei Li}
\ead{leili@math.duke.edu}
\author[mymainaddress,mysecondaryaddress]{Jian-Guo Liu}
\ead{jliu@phy.duke.edu}
\author[mysecondaryaddress1]{Lizhen Wang\corref{mycorrespondingauthor}}

\cortext[mycorrespondingauthor]{Corresponding author}
\ead{wanglz123@hotmail.com}

%\fntext[myfootnote0]{Partially supported by KI-Net NSF RNMS 11-07444 and NSF DMS grant No. 1514826.}
%\fntext[myfootnote]{Partially supported by NSFC (Grant No. 11771352, 11631007, 11571279, 11371293).}
\address[mymainaddress]{Department of Physics,
Duke University, Durham, NC, 27708, USA}
\address[mysecondaryaddress]{Department of Mathematics,
Duke University, Durham, NC, 27708, USA}
\address[mysecondaryaddress1]{Center for Nonlinear Studies, School of Mathematics, Northwest University, Xi'an, Shaanxi Province, 710069, China }

\begin{abstract}
 This paper investigates Cauchy problems for nonlinear fractional time-space generalized Keller-Segel equation $^c_0D_t^\beta\rho+(-\triangle)^{\frac{\alpha}{2}}\rho+\nabla\cdot(\rho B(\rho))=0$, where Caputo derivative $^c_0D_t^\beta\rho$ models memory effects in time, fractional Laplacian $(-\triangle)^{\frac{\alpha}{2}}\rho$ represents L\'evy diffusion and $B(\rho)=-s_{n,\gamma}\int_{R^n}\frac{x-y}{|x-y|^{n-\gamma+2}}\rho(y)dy $ is the general potential with a singular kernel which takes into account the long rang interaction. We first establish $L^r-L^q$ estimates and weighted estimates of the fundamental solutions $(P(x,t), Y(x,t))$ (or equivalently, the solution operators $(S_\alpha^\beta(t), T_\alpha^\beta(t))$). Then, we prove the existence and uniqueness of the mild solutions when initial data are in $L^p$ spaces, or the weighted spaces. Similar to Keller-Segel equations, if the initial data are small in
 critical space $L^{p_c}(\mathbb{R}^n)$ ($p_c=\frac{n}{\alpha+\gamma-2}$), we construct the global existence. Furthermore, we prove the $L^1$ integrability and integral preservation when the initial data are in $L^1(\mathbb{R}^n)\cap L^p(\mathbb{R}^n)$ or $L^1(\mathbb{R}^n)\cap L^{p_c}(\mathbb{R}^n)$. Finally, some important properties of the mild solutions including the nonnegativity preservation, mass conservation and blowup behaviors are established.
\end{abstract}

\end{frontmatter}

%\linenumbers

\section{Introduction}

Fractional derivatives \cite{skm93,kst06,DK,gly2015,liliu17} are employed to describe the nonlocal effects in time and space.  They are integro-differential operators generalizing the definition of integer order derivative to fractional orders, and have been used to deal with numerous application in areas such as physics, hydrology, biomedical engineering, control theory, to name a few \cite{N,KS,lll17,DCL,W,SBM, liliu17discrete}. Time fractional derivatives \cite{D,skm93,kst06,liliu17} are usually applied to model the ubiquitous memory effects. The theory of time fractional differential equations, especially time fractional ODEs, has been developed by many authors \cite{kst06,DK,fllx17,llljg17}.
Both time and spatial fractional derivatives \cite{N,V2,DCL} can be used for anomalous diffusion or dispersion when a particle plume spreads at a rate inconsistent with the Brownian motion models. The appearance of spatial fractional derivatives in diffusion equations are exploited for  macroscopic description of transport and often lead to superdiffusion phenomenon. Time fractional derivatives are usually connected with anomalous subdiffusion, where a cloud of particles spread more slowly than a classical diffusion \cite{MS,CKK}, because  particle sticking and trapping phenomena ordinarily display power-law behaviors.

There have been a lot of works investigating fractional partial differential equation (see \cite{V2,HL,eidelman2004,kv14,Taylor,RZ,kemppainen2017} for example).  Huang and Liu \cite{HL} studied the uniqueness and stability of nonlocal Keller-Segel equations by considering a self-consistent stochastic process driven by rotationally invariant $\alpha$-stable L\'evy process. Taylor \cite{Taylor} constructed the formulas and estimates for the solution to nonhomogeneous time fractional diffusion equations $ ^c_0D_t^\beta u+Au-q(t)=0$ where $A$ is a positive self-adjoint operator. Zacher \cite{RZ} considered the regularity of weak solutions to linear diffusion equations with Riemann-Liouville time fractional derivative in a bounded domain in $\mathbb{R}^n$.  Kemppainen, Siljander and Zacher \cite{kemppainen2017} performed a careful analysis of the large-time behaviors for fully nonlocal diffusion equations. Allen, Caffarelli and Vasseur \cite{ACV0,ACV} discussed porous medium flows and parabolic problems with fractional time derivative of Caputo-type.

In this paper, we focus on Cauchy problems of the following nonlinear time-space fractional diffusion equations (NFDE):
\begin{equation}\label{generalized Keller-Segel}
\left\{
  \begin{aligned}
  &^c_0D_t^\beta\rho+(-\triangle)^{\frac{\alpha}{2}}\rho+\nabla\cdot(\rho B(\rho))=0  && \mbox{in } (x,t)\in \mathbb{R}^n\times(0,\infty),\\
  &\rho(x,0)=\rho_0(x),\\
  \end{aligned}
     \right.
\end{equation}
where $0<\beta<1,1<\alpha\leq2$. $ ^c_0D_t^\beta$ is Caputo fractional derivative operator of order $\beta$. Caputo derivative was first introduced in \cite{D} and is more suitable for the initial-value problem compared with the Riemann-Liouville fractional derivative. There are many recent definitions of the Caputo derivative in the literature listed to generalize the traditional definition \cite{gly2015,ACV,bernardis2016,liliu17,llljg17}. We will use the definition introduced in \cite{liliu17,llljg17} because of the theoretic convenience (see also Section \ref{sec:preliminary} for the detailed introduction). When the function $v$ is absolutely continuous in time, the definition in  \cite{liliu17,llljg17} reduces to the traditional form:
 \begin{equation}\label{Caputofractionalderivative}
 ^c_0D_t^\beta v(t)=\frac{1}{\Gamma(1-\beta)}\int_0^t(t-s)^{-\beta}\dot{v}(s)ds,
 \end{equation}
where $\Gamma$ is the Gamma function and $\dot{v}(t)$ is the first order integer derivative of function $v(t)$ with respect to independent variable $t$.
According to Chapter V in \cite{Stein}, the nonlocal operator $(-\triangle)^{\frac{\alpha}{2}}$, known as the Laplacian of order $\frac{\alpha}{2}$, is given by means of the Fourier multiplier
\[
(-\triangle)^{\frac{\alpha}{2}}\rho(x):=\mathcal{F}^{-1}\big(|\xi|^{\alpha}\hat{\rho}(\xi)\big)(x).
\]
where
\begin{gather}\label{eq:FourierTrans}
\hat{\rho}(\xi)=\mathcal{F}\big(\rho(x)\big)=\int_{\mathbb{R}^n}\rho(x)e^{-ix\cdot\xi}\,dx
\end{gather}
is the Fourier transformation of $\rho(x)$.
For $\gamma\in(1,n],$ $ n\geq2$ and some constant $s_{n,\gamma}>0$, the linear vector operator $B$ can be formally represented as $B(\rho)=\nabla((-\bigtriangleup)^{-\frac{\gamma}{2}}\rho)$ and is explicitly expressed with convolution of a singular kernel as follows\\
\begin{equation}\label{eq:B}
B(\rho)(x)=-s_{n,\gamma}\int_{\mathbb{R}^n}\frac{x-y}{|x-y|^{n-\gamma+2}}\rho(y)dy,
\end{equation}
which is the attractive kernel as \cite{LRZ} pointed out.

Model \eqref{generalized Keller-Segel} generalizes the well known parabolic-elliptic Keller-Segel model \cite{JL,horstmann2003,BDP}:
\begin{gather}\label{eq:classicalks}
\left\{
\begin{split}
& \partial_t\rho=\Delta\rho-\nabla\cdot(\rho\nabla c),\\
& -\Delta c=\rho.
\end{split}
\right.
\end{gather}
Here, $\rho(x,t)$ represents the density of bacteria cells and $c$ represents the density of chemical substance.  The original parabolic-parabolic Keller-Segel model was first introduced in \cite{keller1970} for chemotactic migration processes, while the parabolic-elliptic model was introduced due to the fact that the diffusion coefficient of the chemical substance is very large \cite{JL}. The equation $-\Delta c=\rho$ models the fact that the bacteria generate chemical substance and the  term $-\nabla \cdot(\rho\nabla c)$ comes from  chemotaxis (the bacteria move according to the chemical gradient). Hence, $-\nabla \cdot(\rho\nabla c)$ is the aggregation term and the two terms on the right hand side of the first equation compete with each other. The aggregation term can not be bounded in all cases, and the destabilizing effect indeed causes the solution to blow up in finite time for some cases \cite{JL,nagai2001}. In  the classical work of J\"ager and Luckhaus \cite{JL}, they studied the blowup of the Keller-Segel system based on a certain comparison principle for the radially symmetric solutions. This technique has recently been modified for much larger classes including refined chemotaxis models  \cite{taowinkler17,bellomo2017}. Another excellent strategy to prove blowup relies on the moment method dating back to \cite{nagai2001} and later for more general systems in \cite{BK}.  The moment method seems to be restricted to parabolic-elliptic system, and for the techniques regarding fully parabolic-parabolic system, one can refer to \cite{herrero1997,winkler2013}.

Our model \eqref{generalized Keller-Segel} describes the biological phenomenon chemotaxis with both anomalous diffusion and memory effects.  Formally, \eqref{generalized Keller-Segel} is just to replace the time derivative in the Keller-Segel model \eqref{eq:classicalks} with the Caputo fractional derivative and the Laplacian in the second equation with $\gamma/2$-fractional Laplacian, when the memory effect and nonlocality are concerned. Our model retains some interesting and essential features of the Keller-Segel model \eqref{eq:classicalks}. Indeed, in Section  \ref{sec:blowup}, we show that the aggregation term still causes blowup when the mass is large and concentrated for the indices we consider. Since our system is the generalization of the parabolic-elliptic system, the method we use for blowup is the moment method.

With the usual time derivative where no memory is concerned, the generalized Keller-Segel equation for chemotaxis with anomalous diffusion has been studied by several authors. When $\gamma=2$, Escudero \cite{E} constructed global in time solutions for the fractional diffusion with $1<\alpha\le 2$.  In two dimensional space ($n=2$) and $\gamma=2$, Biler and Wu \cite{BW} investigated the Cauchy problem with initial data $u_0$ in critical Besov spaces $\dot{B}^{1-\alpha}_{2,r}(\mathbb{R}^2)$ for $r\in[1,\infty]$ and $1<\alpha<2$. Li, Rodrigo and Zhang \cite{LRZ} proved the wellposedness, continuation criteria and smoothness of local solutions. For general $\gamma$ and $\alpha$, Biler and Karch \cite{BK} studied the local and global existence of mild solutions in $C([0, T], L^p(\mathbb{R}^n))$,  and blowup behaviors.

For the model with memory effects and anomalous diffusion \eqref{generalized Keller-Segel},  first of all, in Proposition \ref{pro:Lpofevoloperator}, \ref{pro:timecontinuity} and \ref{pro:weightedfund} , we construct the $L^r-L^q$ estimates, the time continuity and weighted estimates of the solution operators $(S_{\alpha}^{\beta}(t), T_{\alpha}^{\beta}(t))$, using the results of the asymptotic behavior of the fundamental solutions $(P(x,t), Y(x,t))$ to fully nonlocal diffusion equations established in \cite{eidelman2004,kemppainen2017}, listed in Lemma \ref{lmm:asymp}, Proposition \ref{pro:fundP} and Proposition \ref{pro:fundY}. With the help of above estimates, in Theorem \ref{thm:localexis}, \ref{thm:globalexis} and \ref{thm:weighted}, we prove the existence and uniqueness of the mild solutions to \eqref{generalized Keller-Segel} with the initial data in $L^p$ spaces and weighted spaces.  The main results regarding $L^p$ spaces are as following:
\begin{thm*}[Theorem \ref{thm:localexis}]
Suppose $n\ge 2$, $0<\beta<1$, $1<\alpha\le 2$ and $1<\gamma\leq n$. Let $p\in (p_c,\infty)\cap [\frac{2n}{n+\gamma-1}, \frac{n}{\gamma-1})$. Then, for any $\rho_0\in L^p(\mathbb{R}^n)$, there exists $T>0$ such that the equation \eqref{generalized Keller-Segel} admits a unique mild solution in $C([0, T]; L^p(\mathbb{R}^n))$ with initial value $\rho_0$ in the sense of Definition \ref{def:mild}. Define the largest time of existence
\[
T_b=\sup\{T>0: \text{\eqref{generalized Keller-Segel} has a unique mild solution in }C([0, T]; L^p(\mathbb{R}^n))\}.
\]
Then if $T_b<\infty$, we have $\limsup_{t\to T_b^-}\|\rho(\cdot, t)\|_{p}=+\infty$.
\end{thm*}

\begin{thm*}[Theorem \ref{thm:globalexis}]
Suppose $n\ge 2$, $0<\beta<1$, $1<\alpha\le 2$ and $1<\gamma\leq n$. Let $\nu=\infty$ if $2(\alpha+\gamma-2)\beta-\alpha\le 0$ or $\nu=\frac{2n\beta}{2(\alpha+\gamma-2)\beta-\alpha}$ if $2(\alpha+\gamma-2)\beta-\alpha> 0$.
Then, whenever $(p_c,\nu)\cap [\frac{n}{n-\alpha+1}, \frac{n}{\gamma-1})$ is nonempty, for any $p\in (p_c,\nu)\cap [\frac{n}{n-\alpha+1}, \frac{n}{\gamma-1})$, there exists $\delta>0$ such that for all $\rho_0\in L^{p_c}(\mathbb{R}^n)$ with
$\|\rho_0\|_{p_c}\le \delta$,  the equation \eqref{generalized Keller-Segel} admits a mild solution $\rho\in C([0, \infty); L^{p_c}(\mathbb{R}^n))$ with initial value $\rho_0$ in the sense of Definition \ref{def:mild}, satisfying
\begin{gather}
\|\rho(t)\|_{p_c}\le 2\delta, \forall t>0,
\end{gather}
and $\rho\in C((0,\infty), L^p(\mathbb{R}^n))$.
Further, the solution is unique in
\begin{gather}
X_T:=\Big\{\rho\in C([0, T]; L^{p_c}(\mathbb{R}^n))\cap C((0,T], L^p(\mathbb{R}^n)): \|\rho\|_{p_c,p; T}<\infty \Big\},~T\in (0, \infty).
\end{gather}
\end{thm*}
Here, the critical $p_c$ is given by $p_c=\frac{n}{\alpha+\gamma-2}$ (see Section \ref{sec:exisuniq} to see why $L^{p_c}$ is critical). These results are based on a fixed point lemma for a bilinear operator in Banach space (see Lemma \ref{lmm:basicexis}).

In the framework of mild solutions, some basic properties such as nonnegativity preservation and mass conservation are nontrivial. We prove the $L^1$ integrability and integral preservation property of the mild solutions in Theorem \ref{thm:LpL1} and \ref{thm:L1critical} when initial data belong to $L^1(\mathbb{R}^n)\cap L^p(\mathbb{R}^n)$ spaces or $L^1(\mathbb{R}^n)\cap L^{p_c}(\mathbb{R}^n)$. Interestingly, if $\rho_0\in L^1(\mathbb{R}^n)\cap L^p(\mathbb{R}^n)$, we can prove the existence of mild solutions for larger range of $p$ values in Theorem \ref{thm:LpL1} compared with Theorem \ref{thm:localexis}. We then establish nonnegativity preservation (and therefore mass conservation) in Section \ref{sec:positivity}. The main result regarding nonnegativity reads:
\begin{thm*}[Theorem \ref{thm:positivity}]
In Theorem \ref{thm:localexis} (or Theorem \ref{thm:globalexis}, Theorem \ref{thm:LpL1}, Theorem \ref{thm:L1critical}, Theorem \ref{thm:weighted}), if we also have  $\rho_0\ge 0$, then for all $t$ in the interval of existence we have
\begin{gather}
\rho(x, t)\ge 0.
\end{gather}
\end{thm*}
The nonnegativity preservation is based on the following important properties (see Corollary \ref{cor:lemfrac} and equation \eqref{5.9})
\begin{gather*}
\langle (-\Delta)^{\frac{\alpha}{2}}\rho, \rho^+\rangle\ge 0,\\
\int_{\mathbb{R}^n} (_0^cD_t^{\beta}\rho) \rho^-\,dx\le -\frac{1}{2} (_0^cD_t^{\beta}\|\rho^-\|_2^2),
\end{gather*}
where $\rho^+=\max(\rho, 0)=\rho \vee 0$ and $\rho^-=-\min(\rho, 0)=-\rho\wedge 0$. The rigorous proof is achieved by utilizing some approximating sequences to gain required regularity.

In Section \ref{sec:blowup}, we investigate blowup behaviors of system \eqref{generalized Keller-Segel}  using the $\nu$-moment  ($1<\nu<\alpha$) following the method used in \cite{nagai2001,BK}. The main results are as follows:
\begin{thm*}[Theorem \ref{thm:blowup}]
Assume $n\ge 2$, $0<\beta<1, 1<\alpha\le 2, 1<\gamma\leq n$. Assume $\rho_0\ge 0$ and also the conditions in Theorem \ref{thm:LpL1} (or Theorem \ref{thm:weighted}) hold to ensure the existence of mild solutions. If one of the following conditions are satisfied,
\begin{enumerate}[(i)]
\item For $\alpha=2, \gamma=n,\nu=2$, if $\rho_0\in L^1(\mathbb{R}^n, (1+|x|^{\nu})dx)$ so that
\[
\|\rho_0\|_1>\frac{2n}{s_{n,\gamma}}.
\]

\item For some $\nu\in (1, \alpha)$ when $\alpha<2$  or $\nu\in (1, 2]$ when $\alpha=2$, with $\frac{n-\gamma+2}{\nu}>1$, if $\rho_0\in L^1(\mathbb{R}^n, (1+|x|^{\nu})dx)$ and
\[
\|\rho_0\|_1>M^*,~\int_{\mathbb{R}^n} |x|^{\nu}\rho_0(x)\,dx<\delta,
\]
for certain constants $M^*(\nu,n,\alpha,\gamma)$ and $\delta(\nu,n, \alpha,\gamma)$.

\item Suppose $\alpha+\gamma<n+2$ (or $p_c>1$) and for some $\nu\in (1, \alpha)$ when $\alpha<2$  or $\nu\in (1, 2]$ when $\alpha=2$. If $\rho_0\in L^1(\mathbb{R}^n, (1+|x|^{\nu})dx)$ satisfying
\begin{equation}\label{BlowupC}
\frac{\int_{\mathbb{R}^n}|x|^\nu \rho_0(x)dx}{\int_{\mathbb{R}^n}\rho_0(x)dx}\leq \chi\left(\int_{\mathbb{R}^n}\rho_0(x)dx\right)^{\frac{\nu}{n+2-\alpha-\gamma}},
\end{equation}
where $\chi=\delta (M^*)^{-1+\frac{\nu}{\alpha+\gamma-2-n}}$ ($M^*$, $\delta$ are the constants in (ii)).
\end{enumerate}
then the mild solution of \eqref{generalized Keller-Segel} will blow up in a finite time.
\end{thm*}

The rest of this paper is organized as follows. In Section \ref{sec:preliminary}, we recall some notations, definitions and known results needed later. In Section \ref{sec:fundamentalest}, we construct the $L^r-L^q$ estimates, the time continuity and weighted estimates of the fundamental solutions and solution operators to \eqref{generalized Keller-Segel}. In Section \ref{sec:exisuniq}, we establish the existence and uniqueness of the mild solutions to NFDE \eqref{generalized Keller-Segel} by applying the results derived in Section \ref{sec:fundamentalest}, and show that the integral of the mild solution is  preserved. The nonnegativity preservation property and conservation of mass are discussed in Section \ref{sec:positivity}. In Section \ref{sec:blowup}, the blowup criterion for \eqref{generalized Keller-Segel} are established with the similar approach as the one used in \cite[Theorem 2.3]{BK}.
%================================================================================================================================
\section{Notations and Preliminaries}\label{sec:preliminary}
In this section, first we provide the definition of Caputo derivative based on a convolution group (introduced  in \cite{liliu17,llljg17}) and show some properties related to Caputo derivatives. Then we define the mild solution to \eqref{generalized Keller-Segel}, introduce the notations of some functional spaces and recall some results which will be used later.

\subsection{Caputo derivatives based on a convolution group}

To introduce the generalized definition of Caputo derivatives, let us first present the following notion of limit:
\begin{defn}\label{def:limit}
Let $B$ be a Banach space. For a locally integrable function $u\in L_{\loc}^1((0, T); B)$, if there exists $u_0\in B$ such that
\begin{gather}
\lim_{t\to 0^+}\frac{1}{t}\int_0^t\|u(s)-u_0\|_Bds=0,
\end{gather}
we call $u_0$ the right limit of $u$ at $t=0$, denoted by $u(0+)=u_0$. Similarly, we define $u(T-)$ to be the constant $u_T\in B$ such that
\begin{gather}
\lim_{t\to T^-}\frac{1}{T-t}\int_{t}^T\|u(s)-u_T\|_Bds=0.
\end{gather}
\end{defn}
As in \cite{liliu17}, we use the following distributions $\{g_\beta\}$ as the convolution kernels for $\beta>-1$:
\begin{gather*}
g_{\beta}(t) :=\displaystyle
\begin{cases}
\frac{\theta(t)}{\Gamma(\beta)}t^{\beta-1},& \beta>0,\\
\delta(t), &\beta=0,\\
\frac{1}{\Gamma(1+\beta)}D\left(\theta(t)t^{\beta}\right), & \beta\in (-1, 0).
\end{cases}
\end{gather*}
Here $\theta(t)$ is the standard Heaviside step function and $D$ means the distributional derivative. $g_{\beta}$ can also be defined for $\beta\le -1$ (see \cite{liliu17}) so that these distributions form a convolution group $\mathscr{C}=\{g_{\beta}: \beta\in\mathbb{R}\}$ and consequently we have
\begin{gather}\label{eq:group}
g_{\beta_1}*g_{\beta_2}=g_{\beta_1+\beta_2},~~\forall \beta_1,\beta_2\in\mathbb{R},
\end{gather}
where the convolution between distributions with one-sided bounded supports can be defined as \cite[Chap. 1]{gs64}.  Correspondingly, we have the time-reflected group
\begin{gather*}
\tilde{\mathscr{C}} :=\{\tilde{g}_{\alpha}: \tilde{g}_{\alpha}(t)=g_{\alpha}(-t), \alpha\in \mathbb{R}\}.
\end{gather*}
Clearly, $\supp \tilde{g}\subset(-\infty, 0]$ and for $\gamma\in (0, 1)$, the following equality is true
\begin{gather}
\tilde{g}_{-\gamma}(t)=-\frac{1}{\Gamma(1-\gamma)}D(\theta(-t)(-t)^{-\gamma})=-D\tilde{g}_{1-\gamma}(t),
\end{gather}
where $D$ represents the distributional derivative on $t$.

To define the weak Caputo derivatives valued in general Banach spaces, we first introduce the definition of the right Caputo derivatives of test functions:
\begin{defn}[\cite{llljg17}]\label{def:rightcaputo}
Let $0<\beta<1$. Consider $u\in L_{\loc}^1((-\infty, T); \mathbb{R})$ such that $u$ has a left limit $u(T-)$ at $t=T$ in the sense of Definition \ref{def:limit}. The $\beta$-th order right Caputo derivative of $u$ is a distribution in $\mathscr{D}'(\mathbb{R})$ with support in $(-\infty, T]$, given by
\[
\tilde{D}_{c;T}^{\beta}u :=\tilde{g}_{-\beta}*(\theta(T-t)(u(t)-u(T-))).
\]
\end{defn}

We now introduce the definition of Caputo derivatives for mappings into Banach spaces:
\begin{defn}[\cite{llljg17}]\label{def:weakcap}
Let $B$ be a Banach space and $u\in L_{\loc}^1([0, T); B)$. Let $u_0\in B$. We define the weak Caputo derivative of $u$ associated with initial data $u_0$, denoted by $_0^cD_t^{\beta}u$, to be a bounded linear functional from $C_c^{\infty}(-\infty, T)$ to $B$ such that for any test function $\varphi \in
C_c^{\infty}((-\infty, T); \mathbb{R}) $,
\begin{gather}
\langle _0^cD_t^{\beta}u, \varphi \rangle :=\int_{-\infty}^T (u-u_0)\theta(t) (\tilde{D}_{c; T}^{\beta}\varphi)\, dt
=\int_0^T (u-u_0)\tilde{D}_{c; T}^{\beta}\varphi \, dt.
\end{gather}
We call the weak Caputo derivative $_0^cD_t^{\beta}u$ associated with initial value $u_0$ the Caputo derivative of $u$ (still denoted as $_0^cD_t^{\beta}u$) if $u(0+)=u_0$ in the sense of Definition \ref{def:limit} under the norm of the underlying Banach space $B$.
\end{defn}

The so-defined Caputo derivatives have the following properties
\begin{lem}[\cite{llljg17}]\label{lmm:defconsistency}
$^c_0D_t^{\beta}u$ is supported in $[0, T)$. In other words, for any $\varphi\in C_c^{\infty}(-\infty, T)$ that is zero on $[0, T)$, the action of $^c_0D_t^{\beta}u$ on $\varphi$ is zero. In addition:
\begin{enumerate}[(i)]
\item
If $B=\mathbb{R}^n$, then the Caputo derivative is consistent with the one in \cite{liliu17}, given by
\[
^c_0D_t^{\beta}u=g_{-\beta}*((u-u_0)\theta(t)).
\]
\item
If $u: (0, T)\to B$ is absolutely continuous, then
\[
^c_0D_t^{\beta}u=\frac{1}{\Gamma(1-\beta)}\int_0^t\frac{\dot{u}(s)}{(t-s)^{\beta}} \in L^1((0, T); B).
\]
\item If $u\in C([0, T), B)\cap C^1((0, T), B)$, then  for $t\in (0, T)$, we have
\begin{equation}\label{Integrationby}
\Gamma(1-\beta)(^c_0D_t^\beta u(t))=\frac{u(t)-u(0)}{t^\beta}+\beta\int_0^t\frac{u(t)-u(s)}{(t-s)^{\beta+1}}ds.
\end{equation}
\end{enumerate}
\end{lem}

\begin{rem}
If $T<\infty$, $g_{-\gamma}*u$ should be understood as the restriction of the convolution onto $\mathscr{D}'(-\infty,T)$. One can refer to \cite{liliu17} for the technical details. Actually, if one defines the convolution between $g_{\alpha}$ and $u$ suitably, the claim in Part (i)
of Lemma \ref{lmm:defconsistency} still holds for general Banach spaces.
\end{rem}

In general, $_0^cD_t^{\beta}u$ is an abstract functional from $C_c^{\infty}((-\infty, T); \mathbb{R})$ to $B$. We say $_0^cD_t^{\beta}u \in L_{\loc}^1([0, T); B)$ if there exists a function $f\in L_{\loc}^1([0, T); B)$ such that for any $\varphi \in C_c^{\infty}((-\infty, T); \mathbb{R})$, we have
\begin{gather}
\langle _0^cD_t^{\beta}u, \varphi \rangle =\int_0^T f(t)\varphi(t)\, dt.
\end{gather}
In this case, we will identify $_0^cD_t^{\beta}u$ with $f$.  With this notion, we have the following claims from \cite{liliu17,llljg17}:
\begin{lem}\label{lmm:volterra}
Let $\beta \in (0, 1)$.
\begin{enumerate}[(i)]
\item Let $u\in L_{\loc}^1((0, T); \mathbb{R})$. Assume that the weak Caputo derivative for an assigned initial value $u_0\in \mathbb{R}$ is $_0^cD_t^{\beta}u$. As linear functionals on $C_c^{\infty}((-\infty, T); \mathbb{R})$, we have
\begin{gather}
(u-u_0)\theta(t)=g_{\beta}*( _0^cD_t^{\beta}u).
\end{gather}

\item If $ f(t):=(_0^cD_t^{\beta}u)\in L_{\loc}^1([0, T); B)$, then
\[
u(t)=u_0+\frac{1}{\Gamma(\gamma)}\int_0^t(t-s)^{\gamma-1} f(s)\, ds, ~a.e.~\text{on }(0, T).
\]
where the integral is understood as the Lebesgue integral.
\end{enumerate}
\end{lem}
The above claims follow from the properties of the underlying convolution group $\mathscr{C}$. The generalized definition in \cite{liliu17,llljg17} has the following theoretic convenience: it only requires the functions to be locally integrable with a certain initial value, which allows the study of solutions of fractional PDEs in very weak sense. There is no requirement on the time regularity of the functions to obtain the claims in  Lemma \ref{lmm:volterra}.

Using Lemma \ref{lmm:volterra}, we are able to conclude the following results:
\begin{lem}
Let $0<\beta<1,T>0$. Suppose $u\in L_{\loc}^1(0, T; B)$ and there is an  assignment of initial value $u_0\in B$ such that $f:= (^c_0D_t^{\beta}u) \in L^p(0,T)$ with $\frac{1}{\alpha}<p\leq\infty$. Then $u$ is continuous and for any $s,t\in[0,T]$ with $s<t$, we have
\begin{equation}\label{estimateFED1}
|u(t)-u(s)|\leq \frac{2}{\Gamma(\beta)}\left(\frac{p-1}{\beta p-1}\right)^{1-\frac{1}{p}}(t-s)^{\beta-\frac{1}{p}}\|f\|_p .
\end{equation}
\end{lem}

\begin{proof}
{\it Case 1. $1<p<\infty$.} For $s<t, 0<\beta<1$, one has
\begin{eqnarray}
\begin{aligned}\label{Lp}
&|u(t)-u(s)|\\
=&
\frac{1}{\Gamma(\beta)}\left|\int_0^s\Big((t-\tau)^{\beta-1}-(s-\tau)^{\beta-1}\Big)f(\tau)d\tau+\int_s^t(t-\tau)^{\beta-1}f(\tau)d\tau\right|\\
\leq&\frac{1}{\Gamma(\beta)}\left(\int_0^s\Big((s-\tau)^{\beta-1}-(t-\tau)^{\beta-1}\Big)|f(\tau)|d\tau+\int_s^t(t-\tau)^{\beta-1}|f(\tau)|d\tau\right)\\
\equiv&\frac{1}{\Gamma(\beta)}(I_1+I_2).
\end{aligned}
\end{eqnarray}
It follows from H\"{o}lder's inequality with $\frac{1}{q}+\frac{1}{p}=1$ that
\begin{align*}
I_1\leq\Big(\int_0^s\big((s-\tau)^{\beta-1}-(t-\tau)^{\beta-1}\big)^qd\tau\Big)^{\frac{1}{q}}\|f\|_p.
\end{align*}
Substitution $s-\tau=\xi(t-s)$ gives that
\begin{align*}
  \left(\int_0^s\Big((s-\tau)^{\beta-1}-(t-\tau)^{\beta-1}\Big)^q d\tau\right)^\frac{1}{q}=(t-s)^{\frac{1}{q}+(\beta-1)}\left(\int_0^\frac{s}{t-s}\Big(\xi^{\beta-1}-(1+\xi)^{\beta-1}\Big)^q d\xi\right)^\frac{1}{q}\!\!.
\end{align*}
If $(\beta-1)q+1>0$, which is equivalent to $p>\frac{1}{\beta}$, it follows from the elementary inequality $(a-b)^r\leq a^r-b^r$ for all $a\geq b \geq0$ and $r\geq1$ that
\begin{align*}
  &\int_0^\frac{s}{t-s}\Big(\xi^{\beta-1}-(1+\xi)^{\beta-1}\Big)^q d\xi
  \leq \int_0^\frac{s}{t-s}\Big(\xi^{(\beta-1)q}-(1+\xi)^{(\beta-1)q}\Big)d\xi\\
  =&\frac{1}{(\beta-1)q+1}\left(\left(\frac{s}{t-s}\right)^{(\beta-1)q+1}-\left(1+\frac{s}{t-s}\right)^{(\beta-1)q+1}+1\right)\\
  &\leq \frac{1}{(\beta-1)q+1}.
\end{align*}
Consequently,
\begin{align}
  I_1\leq(t-s)^{\frac{1}{q}+(\beta-1)}\Big((\beta-1)q+1\Big)^{-\frac{1}{q}}\|f\|_p.
\end{align}
Applying H\"{o}lder's inequality again and then computing the integral directly, one has
\begin{eqnarray}
\begin{aligned}\label{tao}
I_2\leq \Big(\int_s^t\big((t-\tau)^{(\beta-1)q}d\tau\Big)^{\frac{1}{q}}\|f\|_p=(t-s)^{\frac{1}{q}+(\beta-1)}\Big((\beta-1)q+1\Big)^{-\frac{1}{q}}\|f\|_p.
\end{aligned}
\end{eqnarray}
Therefore, we have from \eqref{Lp}-\eqref{tao} and $\frac{1}{q}+\frac{1}{p}=1$ that \eqref{estimateFED1} holds in this case.

{\it Case 2. $p=\infty$.}
With similar argument as \eqref{Lp}, one deduce that
\begin{eqnarray}
\begin{aligned}\label{Lp2}
&|u(t)-u(s)|\\
&\leq\frac{1}{\Gamma(\beta)}\left(\int_0^s\Big((s-\tau)^{\beta-1}-(t-\tau)^{\beta-1}\Big)d\tau
\|f\|_\infty+\int_s^t(t-\tau)^{\beta-1}d\tau\|f\|_\infty\right)\\
&=\frac{1}{\Gamma(1+\beta)}\Big(s^\beta-t^\beta+2(t-s)^\beta\Big)\|f\|_\infty\\
&\leq \frac{2}{\Gamma(1+\beta)}(t-s)^\beta\|f\|_\infty.
\end{aligned}
\end{eqnarray}
Hence \eqref{estimateFED1} with $p=\infty$ holds.
\end{proof}

The following comparison principle for Caputo derivative will play an important role in the proof of the nonnegativity and the study of finite time blowup of the solution to (1.1).
\begin{lem}\cite{liliu17}\label{lmm:comparison}
Let $0<\beta<1, T>0$. Assume that $v(t)\in C([0,T]; \mathbb{R})$.
Suppose $f(t, x)$ is a continuous function, locally Lipschitz in $x$, such that $\forall t\ge 0$, $x\le y$ implies $f(t, x)\le f(t, y)$.
If $f(t, v)-(^c_0D_t^{\gamma}v)$ is a nonnegative distribution, then $v\le u$
for $t\in [0, \min(T, T_b))$, where $u$ is the solution to the following ODE
\[
^c_0D_t^{\beta}u=f(t, u),~u(0)=v(0)
\]
and $T_b$ is the largest existence time for $u$.
\end{lem}
\begin{rem}
Indeed, the monotonicity in $x$ for $f(t,x)$ can be removed. The proof will be presented in a forthcoming short note \cite{fllx18note}.
\end{rem}

For real numbers $\beta,\gamma$, the Mittag-Leffler function $E_{\beta,\gamma}: \mathbb{C}\to \mathbb{C}$ is defined by
\begin{equation}\label{ML}
E_{\beta,\gamma}(z)=\sum_{n=0}^{\infty} \frac{z^n}{\Gamma(n\beta+\gamma)}.
\end{equation}
We will denote
\begin{gather}
E_{\beta}(z):=E_{\beta,1}(z).
\end{gather}
If $\beta>0$ and $\gamma>0$, $E_{\beta,\gamma}(z)$ is an entire function.

For function $u\in L_{\loc}^1(0,T; \mathbb{R})$ with $u(0+)=u_0$ and polynomial growth in  time, \cite{liliu17} shows that the Laplace transform of the weak Caputo derivative is still given by:
\begin{gather}
\mathcal{L}(^c_0D_t^{\beta}u)=s^{\beta}\mathcal{L}(u)-u_0s^{\beta-1},
\end{gather}
and the solution to the ODE
\[
^c_0D_t^{\beta}u=\lambda u+f(t),~~u(0)=u_0
\]
is given by
\begin{gather}\label{eq:solutionode}
u(t)=u_0E_{\beta}(\lambda t^{\beta})+\beta \int_0^t s^{\beta-1}
E_{\beta}'(\lambda t^{\beta}) f(t-s)\,ds.
\end{gather}

As a corollary of Lemma \ref{lmm:comparison}, we have
\begin{cor}\label{cor:gronwall}
Let $0<\beta<1, T>0$. Assume that $u(t)\in C([0,T]; \mathbb{R})$.  If $^c_0D_t^{\beta}u \leq a+bu(t)$ for some $b\ge 0$ (the inequality means $a+bu(t)-(^c_0D_t^{\beta}u)$ is a non-negative distribution), then, when $b=0$,
\begin{equation}\label{AKS}
u(t)\leq u(0)+\frac{at^\beta}{\Gamma(\beta+1)};
\end{equation}
when $b\neq 0$,
\begin{equation}\label{IN}
u(t)\leq  u(0)E_\beta(bt^\beta)+\frac{a}{b}\big(E_\beta(bt^\beta)-1\big).
\end{equation}
\end{cor}

\subsection{Basic setup: definitions, notations and preliminary lemmas}

Denote
\begin{gather}
A=(-\triangle)^{\frac{\alpha}{2}}
\end{gather}
and $ E_\beta(-t^\beta A)$ is the linear operator defined by operator calculus. Following \cite{Taylor}, formally taking the Laplace transform of \eqref{generalized Keller-Segel}, we can find that $\rho$ satisfies the following Duhamel type integral equation (similar to \eqref{eq:solutionode}) though the equation \eqref{generalized Keller-Segel} is non-Markovian:
\begin{gather}\label{solutiontoequation}
\begin{split}
 \rho(x,t)&=E_\beta(-t^\beta A)\rho_0-\beta\int_0^t(t-\tau)^{\beta-1}E'_\beta(-(t-\tau)^\beta A)(\nabla\cdot(\rho B (\rho))(\tau)d\tau\\
 &=E_\beta(-t^\beta A)\rho_0-\int_0^t(t-\tau)^{\beta-1}E_{\beta,\beta}(-(t-\tau)^\beta A)(\nabla\cdot(\rho B (\rho))(\tau)d\tau,
 \end{split}
\end{gather}
where we have used the well-known fact:
\[
\beta E_{\beta}'(z)=E_{\beta,\beta}(z).
\]
This formal computation then motivates the definition of the mild solution as follows:
\begin{defn}\label{def:mild}
Let $X$ be a Banach space over space and time. We call $\rho\in X$ is a mild solution to \eqref{generalized Keller-Segel} if $\rho$ satisfies the integral equation \eqref{solutiontoequation} in $X$.
\end{defn}

Let us now clarify the notations for the spaces which will be used later in this paper.
For $1\leq p\leq\infty$, we use $\|u\|_p$ to denote the $L^p$-norm of a Lebesgue measurable function $u$ in  $L^p(\mathbb{R}^n)$ space. Recall that the weak $L^p$ norm is defined by
\begin{gather}
\|u\|_{L^{p,\infty}}=\sup_{\lambda>0}\{\lambda d_u(\lambda)^{\frac{1}{p}}\},~~d_u(\lambda)=|\{x: |u(x)|>\lambda \}|.
\end{gather}

The Hardy-Littlewood-Sobolev inequality for $L^p$ spaces is listed in the following lemma.
\begin{lem}\label{lmm:hls}
\cite[ page 119, Theorem 1]{Stein}
Let $0<\ell<n$, $1<p<q<\infty$ and $\frac{1}{q}=\frac{1}{p}-\frac{\ell}{n}$. Then
\begin{equation}\label{HLS}
\|\int_{R^n}\frac{f(y)}{|x-y|^{n-\ell}}dy\|_q\leq C\|f\|_p,
\end{equation}
holds for all $f\in L^p(\mathbb{R}^n)$.
\end{lem}
Consequently, we have
\begin{cor}\label{cor:hls}
Assume $0\le \lambda<n,p=\frac{2n}{2n-\lambda}$. Then
\begin{equation}
\left|\int_{\mathbb{R}^n}\int_{\mathbb{R}^n}\frac{f(x)f(y)}{|x-y|^\lambda}dxdy\right|\leq C\|f\|_p^2.
\end{equation}
\end{cor}
When $\lambda=0$, it is trivial while for $1<\lambda<n$, it follows from Lemma \ref{lmm:hls}.

For fixed $\nu\geq0$, we define the weighted space $L^\infty_\nu(\mathbb{R}^n)$ as follows
\begin{equation}\label{WLI}
L^\infty_\nu(\mathbb{R}^n)=\{v\in L^\infty(\mathbb{R}^n):\|v\|_{L^\infty_\nu}:= \|(1+|x|)^\nu v(x)\|_{\infty}<\infty\}.
\end{equation}

For $s\in \mathbb{R}, $ $1<p<\infty$, the Sobolev spaces are defined in \cite[Chapter 13.6]{Taylor1}:
\[
 H^{s,p}(\mathbb{R}^n)=\{u\in L^p(\mathbb{R}^n): \mathcal{F}^{-1}((1+|\xi|^2)^{\frac{s}{2}}\hat{u}) \in L^p(\mathbb{R}^n)\},
 \]
 where $\mathcal{F}$ is the Fourier transform and the norm given by
 \[
 \|u\|_{H^{s,p}}:=\|\mathcal{F}^{-1}((1+|\xi|^2)^{\frac{s}{2}}\hat{u})\|_p.
 \]
 The $H^{s,p}$ spaces are also called the Bessel potential spaces and sometimes denoted by $W^{s,p}(\mathbb{R}^n)$.
 The following Sobolev embedding is standard:
\begin{lem}
\cite{Taylor1}
 For $0<sp<n,$ $1< p<\infty$, then
\[
H^{s,p}(\mathbb{R}^n)\subset L^{\frac{np}{{n-sp}}}(\mathbb{R}^n).
\]
\end{lem}

%================================================================================================================================
\section{Estimates of the fundamental solutions}\label{sec:fundamentalest}

Consider the functions $P(x, t)$ and $Q(x, t)$ defined for $1<\alpha\le 2$ and $0<\beta<1$:
\begin{gather}
\begin{split}
\mathcal{F}P(\cdot,t)=E_\beta(-|\xi|^\alpha t^\beta),~~
\mathcal{F}Q(\cdot, t)=E_{\beta,\beta}(-|\xi|^\alpha t^\beta),
\end{split}
\end{gather}
where $\mathcal{F}$ is the Fourier transform defined in \eqref{eq:FourierTrans}.

Now, we define
\begin{gather}\label{eq:defY}
Y(x, t)=t^{\beta-1}Q(x, t).
\end{gather}
It is mentioned in  \cite[Lemma 4.1]{kemppainen2017} that $Y$ is the $1-\beta$ order Riemann-Liouville derivative of $P$.
We pick $A=(-\Delta)^{\frac{\alpha}{2}}$ ($1<\alpha\le 2$), and consider operators $S_{\alpha}^{\beta}(t)$, $T_{\alpha}^{\beta}(t)$ defined by
\begin{gather}\label{defope}
\begin{split}
f(x)\mapsto S_{\alpha}^{\beta}(t)f(x):=E_{\beta}(-t^{\beta}A)f(x)=P(\cdot, t)*f(x),\\
f(x)\mapsto T_{\alpha}^{\beta}(t)f(x):=t^{\beta-1}E_{\beta,\beta}(-t^{\beta}A)f(x)=Y(\cdot, t)*f(x).
\end{split}
\end{gather}
The pair $\{P, Y\}$ is called the fundamental solutions to the Cauchy problem \eqref{generalized Keller-Segel}. In particular,
$S_{\alpha}^{\beta}\rho_0$ is the formal solution to the following initial value problem
\begin{equation}\label{homogenous}
\left\{
  \begin{aligned}
  ^c_0D_t^\beta \rho+(-\triangle)^{\frac{\alpha}{2}} \rho=0,&& \mbox{in } (x,t)\in\mathbb{R}^n\times(0,\infty)\\
  \rho(x,0)=\rho_0(x),\\
  \end{aligned}
     \right.
\end{equation}
while the mild solution in \eqref{solutiontoequation} can be rewritten as
\begin{gather}\label{eq:mildreformulation}
\begin{split}
&\rho(t)=S_{\alpha}^{\beta}(t)\rho_0
-\int_0^t \nabla_x\cdot(T_{\alpha}^{\beta}(t-s) \rho(s)B(\rho)(s))\,ds\\
&=\int_{\mathbb{R}^n} P(x-y, t)\rho_0(y)\,dy-\int_0^t\int_{\mathbb{R}^n}\nabla_x\cdot (Y(t-s, x-y)\rho(y,s)B(\rho)(y,s))\,dyds.
\end{split}
\end{gather}

By extending the results in \cite{KL,kemppainen2017}, we have the following claims
\begin{lem}\label{lmm:PQ}
$P$ and $Q$ are both nonnegative and integrable. In particular, we have
\begin{gather}
\int_{\mathbb{R}^n}P(x, t)\,dx=1,~~~\int_{\mathbb{R}^n}Q(x, t)\, dx=\frac{1}{\Gamma(\beta)}.
\end{gather}
\end{lem}
\begin{proof}
The proof follows from Remark 4.2 in \cite{kemppainen2017}. It is well-known that
\[
s \mapsto E_{\beta}(-t^{\beta}s^{\alpha}),~~s \mapsto E_{\beta,\beta}(-t^{\beta}s^{\alpha})
\]
are completely monotone functions on $\mathbb{R}_+$. Hence, $\mathcal{F}P(\cdot,t)$ and $\mathcal{F}Q(\cdot, t)$ are positive definite on $\mathbb{R}^n$. By Bochner's theorem \cite{rudin2017}, both $P$ and $Q$ are nonnegative.

Note that the Fourier transform evaluated at $\xi=0$ equals the integral of the function. Then, we have
\[
\int_{\mathbb{R}^d}P(x,t)\,dx=E_{\beta}(-0^{\alpha}t^{\beta})=1.
\]
With the fact that $P(\cdot, t)$ is nonnegative, we find that $P(\cdot, t)$ is integrable and the integral values is $1$. Similar results for $Q(\cdot, t)$ follow from the fact
\[
E_{\beta,\beta}(0)=\frac{1}{\Gamma(\beta)}.
\]
\end{proof}

Next, we are going to collect some estimates of these operators, which will be useful for the analysis of time fractional PDEs (though some of them are not used in this paper).

\subsection{Contraction  properties}

The contraction properties follow from the asymptotic behavior of the Mittag-Leffler functions.  We have the following estimates regarding the operators appeared in \eqref{solutiontoequation}:
\begin{lem}
\cite[(8.23),(8.36),(8.38) ]{Taylor}
\begin{enumerate}[(i)]
\item Suppose that $e^{-tA}$ is a contraction semi-group in a Banach space, where $A$ is the generator of the semigroup. Then,
\[
||E_\beta(-t^\beta A)f||_{B}\leq \|f||_B,~~
||E_{\beta,\beta}(-t^\beta A)f||_{B}\leq \frac{1}{\Gamma(\beta)}\|f\|_B.
\]
\item
Let $0<\alpha\leq 2$ and $A=(-\Delta)^{\frac{\alpha}{2}}$.
If $1< p<\infty$ and $\sigma\in(0,1]$, then for $T_0>0$, there exists $C>0$ such that
\begin{equation}\label{Sin11}
||E_\beta(-t^\beta A)f||_{H^{\sigma\alpha,p}}\leq C t^{-\sigma\beta}\|f||_p,~~
||E_{\beta,\beta}(-t^\beta A)f||_{H^{\sigma\alpha,p}}\leq Ct^{-\sigma\beta}\|f||_p,
\end{equation}
uniformly for $t\in (0,T_0]$.
\end{enumerate}
\end{lem}

\subsection{$L^r-L^q$ estimates}\label{sec:Lpfund}
The $L^r-L^q$ estimates follow from the asymptotic study of the fundamental solutions, which have been established in \cite{eidelman2004,kemppainen2017} using the asymptotic behaviors of the so-called Fox $H$-functions:
\begin{lem}\label{lmm:asymp}
Let $0<\beta<1, 1<\alpha\leq2$. We have the following asymptotic estimates for $P$ and $Y$:\\
(1)When $|x|^\alpha t^{-\beta} \geq 1$, there exists $C>0$ such that
\begin{equation}\label{eq:Pfaraway}
|P(x,t)|\leq
\begin{cases}
C|x|^{-n-\alpha}t^\beta, & 1<\alpha<2,\\
Ct^{-\frac{n\beta}{2}}\exp\{-C|x|^{\frac{\alpha}{\alpha-\beta}}t^{-\frac{\beta}{\alpha-\beta}}\}, & \alpha=2,
\end{cases}
\end{equation}
\begin{equation}
|\nabla P(x,t)|\leq
\begin{cases}
C|x|^{-n-\alpha-1}t^\beta, & 1<\alpha<2,\\
Ct^{-\frac{\beta(n+1)}{2}}\exp\{-C|x|^{\frac{\alpha}{\alpha-\beta}}t^{-\frac{\beta}{\alpha-\beta}}\}, & \alpha=2,
\end{cases}
\end{equation}
and that
\begin{equation}
|Y(x,t)|\leq
\begin{cases}
C|x|^{-n-\alpha}t^{2\beta-1}, & 1<\alpha<2,\\
Ct^{-\frac{n\beta}{2}+\beta-1}\exp\{-C|x|^{\frac{\alpha}{\alpha-\beta}}t^{-\frac{\beta}{\alpha-\beta}}\}, & \alpha=2.
\end{cases}
\end{equation}
\begin{equation}\label{eq:gradY1}
|\nabla Y(x,t)|\leq
\begin{cases}
C|x|^{-n-\alpha-1}t^{2\beta-1}, & 1<\alpha<2,\\
Ct^{-\frac{\beta(n+1)}{2}+\beta-1}\exp\{-C|x|^{\frac{\alpha}{\alpha-\beta}}t^{-\frac{\beta}{\alpha-\beta}}\},  & \alpha=2,
\end{cases}
\end{equation}

(2) When $|x|^\alpha t^{-\beta} \leq 1$, there is $C>0$ such that
\begin{equation}\label{eq:Psmall}
|P(x,t)|\leq
\begin{cases}
Ct^{-\frac{n\beta}{\alpha}}, & \alpha>n,\\
C|x|^{-n+\alpha}t^{-\beta}, & \alpha<n,\\
Ct^{-\beta}(1+|\ln|x|^\alpha t^{-\beta}|), & \alpha=n(=2).
\end{cases}
\end{equation}
\begin{equation}\label{Lemma2.88}
|\nabla P(x,t)|\leq C|x|^{-n+\alpha-1}t^{-\beta}.
\end{equation}
and that
\begin{equation}
|Y(x,t)|\leq
\begin{cases}
Ct^{-\frac{n\beta}{\alpha}+\beta-1}, & 2\alpha>n,\\
C|x|^{-n+2\alpha}t^{-\beta-1}, & 2\alpha<n,\\
Ct^{-\beta-1}(1+|\ln|x|^\alpha t^{-\beta}|), & 2\alpha=n.
\end{cases}
\end{equation}
\begin{equation}\label{eq:gradY2}
|\nabla Y(x,t)|\leq
\begin{cases}
t^{-\beta-1}|x|^{-n-1+2\alpha} & n>2\alpha-2,\\
t^{-\beta-1}|x|(1+\log(|x|^{\alpha}t^{-\beta})), & n=2\alpha-2,\\
t^{\beta-1-\frac{\beta(n+2)}{\alpha}}|x| & n<2\alpha-2.
\end{cases}
\end{equation}
\end{lem}

\begin{rem}
In \cite{KL}, in the expression for $P$, the prefactor of the exponential is given by $|x|^{-n}$ in the case $\alpha=2$ and $|x|^{\alpha}t^{-\beta}\ge 1$. This is indeed equivalent to the above estimates. In fact, we introduce $z=|x|^{2}t^{-\beta}$ and $|x|^{-n}\exp(\ldots)=t^{-\frac{n\beta}{2}} (z^n\exp(\ldots))$ which is controlled from above and below by the same exponential with a different constant inside.
\end{rem}
Using Lemma \ref{lmm:asymp}, one can derive the estimates of $\|P\|_p$ and $\|\nabla P\|_p$ (see \cite[Lemma 6.1, Lemma 6.22]{kemppainen2017}). We summarize the results as following.
\begin{prop}\cite{kemppainen2017}\label{pro:fundP}
Suppose $0<\beta<1$ and $1<\alpha\le 2$.

(1).  Set $\kappa_1=\frac{n}{n-\alpha}$ if $n>\alpha$ and $\kappa_1=\infty$ otherwise. Then we have for any $p\in [1, \kappa_1)$, there exists $C>0$ such that
\begin{equation}\label{Lp-estimate}
\|P\|_p\leq C t^{-\frac{n\beta}{\alpha}(1-\frac{1}{p})}.
\end{equation}
If $n<\alpha$ (or $n=1$), \eqref{Lp-estimate} also holds for $p=\kappa_1=\infty$. If $n>\alpha$, for $p=\kappa_1=\frac{n}{n-\alpha}$, \eqref{Lp-estimate} holds only for weak $L^{\kappa_1}$ norm:
\[
\|P\|_{L^{\kappa_1,\infty}}\le C t^{-\beta}.
\]

(2). Let $\kappa_2=\frac{n}{n-\alpha+1}$ if $n>\alpha-1$ and $\kappa_2=\infty$ otherwise. Then for $p\in [1, \kappa_2)$, there is $C>0$ such that
\begin{equation}\label{lpestimate}
\|\nabla P\|_p\leq C t^{-\frac{n\beta}{\alpha}(1-\frac{1}{p})-\frac{\beta}{\alpha}}.
\end{equation}
If $n\le \alpha-1$ (or $n=1, \alpha=2$), \eqref{lpestimate} also holds for $p=\kappa_2=\infty$. For $n>\alpha-1$ and $p=\kappa_2$, \eqref{lpestimate} only holds in weak $L^p$:
\[
\|\nabla P\|_{L^{\kappa_2,\infty}}\le C t^{-\beta}.
\]
\end{prop}

Similarly, we have the estimates for $Y$:
\begin{prop}\label{pro:fundY}
Suppose $0<\beta<1$ and $1<\alpha\le 2$.

(1).  Set $\kappa_3=\frac{n}{n-2\alpha}$ if $n>2\alpha$ and $\kappa_1=\infty$ otherwise. Then we have for any $p\in [1, \kappa_3)$, there exists $C>0$ such that
\begin{equation}\label{Lp-estimateY}
\|Y\|_p\leq C t^{-\frac{n\beta}{\alpha}(1-\frac{1}{p})+\beta-1}.
\end{equation}
If $n<2\alpha$, \eqref{Lp-estimateY} also holds for $p=\kappa_3=\infty$. If $n>2\alpha$, for $p=\kappa_3=\frac{n}{n-2\alpha}$, \eqref{Lp-estimateY} holds only for weak $L^{\kappa_1}$ norm:
\[
\|Y\|_{L^{\kappa_3,\infty}}\le C t^{-\beta-1}.
\]

(2). Let $\kappa_4=\frac{n}{n-2\alpha+1}$ if $n>2\alpha-1$ and $\kappa_4=\infty$ otherwise. Then for $p\in [1, \kappa_4)$, there is $C>0$ such that
\begin{equation}\label{eq:gradYLp}
\|\nabla Y\|_p\leq C t^{-\frac{n\beta}{\alpha}(1-\frac{1}{p})-\frac{\beta}{\alpha}+\beta-1}.
\end{equation}

If $n\le 2\alpha-1$, \eqref{eq:gradYLp} also holds for $p=\kappa_4=\infty$. For $n>2\alpha-1$ and $p=\kappa_4$, \eqref{eq:gradYLp} only holds in weak $L^p$:
\[
\|\nabla Y\|_{L^{\kappa_2,\infty}}\le C t^{-\beta-1}.
\]
\end{prop}
The proof of Proposition \ref{pro:fundY} is similar to the proof of Proposition \ref{pro:fundP}. Part (1) of Proposition \ref{pro:fundY} is the Lemma 6.2 in \cite {kemppainen2017} and part (2) does not appear in \cite{kemppainen2017}. Though the proof is similar, due to the importance of this result, we attach the proof in \ref{app:fund} for completeness.

With the help of Propositions \ref{pro:fundP} and  \ref{pro:fundY}, we obtain the following $L^r-L^q$ estimates regarding the operators $S_{\alpha}^{\beta}(t)$ and $T_{\alpha}^{\beta}(t)$:
\begin{prop}\label{pro:Lpofevoloperator}
 Let $0<\beta<1$ and $1<\alpha\le 2$.  Then,

 (1). The following $L^{\infty}$ estimates hold:
 \begin{gather}
 \begin{split}
 \|S_{\alpha}^{\beta}(t)u\|_{\infty}
 \le \|u\|_{\infty},~~
 \|T_{\alpha}^{\beta}(t)u\|_{\infty}\le \frac{1}{\Gamma(\beta)}t^{\beta-1}\|u\|_{\infty},\\
 \|\nabla S_{\alpha}^{\beta}(t)u\|_{\infty}\le C t^{-\frac{\beta}{\alpha}}\|u\|_{\infty},
 ~~\|\nabla T_{\alpha}^{\beta}(t)u\|_{\infty}\le Ct^{-\frac{\beta}{\alpha}+\beta-1}\|u\|_{\infty}.
 \end{split}
 \end{gather}

 (2). Let $q\in [1,\infty)$. We define  $\theta_1=\frac{qn}{n-q\alpha}$ if $n>q\alpha$ and $\theta_1=\infty$ otherwise.
 Then, for any $r\in [1, \theta_1)$, we have
 \begin{gather}
 \|S_{\alpha}^{\beta}(t)u\|_r \le C t^{-\frac{n\beta}{\alpha}(\frac{1}{q}-\frac{1}{r})}\|u\|_q
 \end{gather}
 If $r=q$, the constant can be chosen to be $1$.  If $n<q\alpha$, then the above also holds for $r=\theta_1=\infty$.

 (3). Let $q\in [1,\infty)$. We define  $\theta_2=\frac{qn}{n-2q\alpha}$ if $n>2n\alpha$ and $\theta_2=\infty$ otherwise.
 Then, for any $r\in [1, \theta_2)$, we have
 \begin{gather}
 \|T_{\alpha}^{\beta}(t)u\|_r \le C t^{-\frac{n\beta}{\alpha}(\frac{1}{q}-\frac{1}{r})+\beta-1}\|u\|_q
 \end{gather}
 If $r=q$, the constant can be chosen as $\frac{1}{\Gamma(\beta)}$.  If $n<2q\alpha$, then the above also holds for $r=\theta_2=\infty$.

 (4). Let $q\in [1,\infty)$. Let $\theta_3=\frac{qn}{n+q(1-\alpha)}$ if $n>q(\alpha-1)$, and $\theta_3=\infty$ otherwise. Then for $r\in [q, \theta_3)$ there is $C>0$ satisfying
\begin{equation}\label{Lp-estimateofsloution2}
\|\nabla S_\alpha^\beta(t)u\|_r\leq C t^{-\frac{n\beta}{\alpha}(\frac{1}{q}-\frac{1}{r})-\frac{\beta}{\alpha}}\|u\|_q.
\end{equation}
If $n<q(\alpha-1)$, the estimate also holds for $r=\theta_3=\infty$.

(5). Let $q\in [1,\infty)$. Let $\theta_4=\frac{qn}{n+q(1-2\alpha)}$ if $n>q(2\alpha-1)$, and $\theta_4=\infty$ otherwise. Then for $r\in [q, \theta_4)$ there is $C>0$ satisfying
\begin{equation}\label{Lp-estimateofsloution3}
\|\nabla T_\alpha^\beta(t)u\|_r\leq C t^{-\frac{n\beta}{\alpha}(\frac{1}{q}-\frac{1}{r})-\frac{\beta}{\alpha}+\beta-1}\|u\|_q.
\end{equation}
If $n<q(2\alpha-1)$, the estimate also holds for $r=\theta_4=\infty$.
\end{prop}

\begin{proof}
Recall that Young's inequality says that if $r,p,q\in [1,\infty]$ satisfying $\frac{1}{r}+1=\frac{1}{p}+\frac{1}{q}$, then
\[
\|f*g\|_r\le \|f\|_p\|g\|_q.
\]

(1). The estimates here follow directly by setting $f=u$, $g=P, Y,\nabla P,\nabla Y$ respectively, $r=p=\infty$ and $q=1$. Recall that $\|P\|_1=1$ and $\|Y\|_1=t^{\beta-1}\|Q\|_1=\frac{1}{\Gamma(\beta)}t^{\beta-1}$. The bounds for $\|\nabla P\|_1$ and $\|\nabla Y\|_1$ follow from \eqref{lpestimate} and \eqref{eq:gradYLp}.

The proofs for (2)-(5) are similar by setting $f=u$, $g=P, Y,\nabla P,\nabla Y$ respectively. In particular, we choose $p\in [1, \kappa_i)$. Then, we have
$1-\frac{1}{p}=\frac{1}{q}-\frac{1}{r}$. Clearly, as long as $\frac{q}{q-1}<\kappa_i$, the estimate holds for all $r\in [q, \infty]$. If $\frac{q}{q-1}>\kappa_i$, which happens only if $\kappa_i<\infty$, then $q<\frac{\kappa_i}{\kappa_i-1}$  and $r<\frac{1}{1/\kappa_i+1/q-1}$.
This then gives the desired results.
\end{proof}
Using the above results, we now establish the time continuity of the fundamental solutions:
\begin{prop}\label{pro:timecontinuity}
Let $\Phi(x):=P(x, 1)$ and $\Psi(x):=Y(x, 1)=Q(x, 1)$. Then, we have
\begin{enumerate}[(i)]
\item $\Phi\in L^p(\mathbb{R}^n)$ when $p\in [1, \kappa_1)$, while $\nabla \Phi\in L^p(\mathbb{R}^n)$ when $p\in [1, \kappa_2)$. $\Psi\in L^p(\mathbb{R}^n)$ when  $p\in [1, \kappa_3)$, while $\nabla\Psi\in L^p(\mathbb{R}^n)$ when $p\in [1, \kappa_4)$.
\item We have the following formulas for $P,Q,Y$:
\begin{gather}
P(x, t)=t^{-\frac{n\beta}{\alpha}}\Phi\left(\frac{x}{t^{\beta/\alpha}}\right),~ Q(x, t)=t^{-\frac{n\beta}{\alpha}}\Psi\left(\frac{x}{t^{\beta/\alpha}}\right), ~ Y(x, t)=t^{-\frac{n\beta}{\alpha}+\beta-1}\Psi\left(\frac{x}{t^{\beta/\alpha}}\right).
\end{gather}
Consequently, $P(x, t)\in C((0, \infty), L^p(\mathbb{R}^n))$ for $p\in [1, \kappa_1)$,
$\nabla P\in C((0,\infty), L^p(\mathbb{R}^n))$ for $p\in [1, \kappa_2)$,
$Y\in  C((0, \infty), L^p(\mathbb{R}^n))$ for $p\in [1, \kappa_3)$
and $\nabla Y\in C((0, \infty), L^p(\mathbb{R}^n))$ for $p\in [1, \kappa_4)$.
\item For any $u\in L^q(\mathbb{R}^n)$ with $q\in [1,\infty)$, $t\mapsto S_{\alpha}^{\beta}u \in C([0,\infty), L^q(\mathbb{R}^n))$.
\end{enumerate}
\end{prop}
\begin{proof}
(i). The claims follow from Proposition \ref{pro:fundP}
and Proposition \ref{pro:fundY}.

(ii).  Since we have $(\mathcal{F}P)(\xi, t)=E_{\beta}(-|\xi|^{\alpha}t^{\beta})$, it is clear that
\[
P(x, t)=\mathcal{F}^{-1}E_{\beta}(-|\xi|^{\alpha}t^{\beta})
=t^{-\frac{n\beta}{\alpha}}\mathcal{F}^{-1}(E_{\beta}(-|\cdot|^{\alpha}))(\frac{x}{t^{\beta/\alpha}})
=t^{-\frac{n\beta}{\alpha}}\Phi\left(\frac{x}{t^{\beta/\alpha}}\right).
\]
Similarly,  we have
\[
Q(x, t)=t^{-\frac{n\beta}{\alpha}}\Psi\left(\frac{x}{t^{\beta/\alpha}}\right),
\]
and hence the formula for $Y(x, t)$ follows from \eqref{eq:defY}.

It is a well-known fact that if $f\in L^q(\mathbb{R}^n)$ with $q\in [1,\infty)$, we have $\|\lambda^nf(\lambda x)-f\|_{L^p}\to 0$ as $\lambda\to 1$ (This can be proved by the standard process of approximating $L^p$ functions with $C_c^{\infty}$ functions). The claims then follow.

(iii). The fact $t\mapsto S_{\alpha}^{\beta}u \in C((0,\infty), L^q(\mathbb{R}^n))$ is  obvious from  the results in (ii). The continuity of $S_{\alpha}^{\beta}u$ at $t=0$ is a standard consequence of mollification, given the expression of $P$ in (ii).

\end{proof}

\subsection{Weighted estimates}\label{WE}
We consider the weighted estimates of the fundamental solutions in the weighted space $L_\nu^\infty(\mathbb{R}^n)$ (see equation \eqref{WLI}).
\begin{prop}\label{pro:weightedfund}
Assume $0<\beta<1$, $1<\alpha\le 2$ and $u_0\in L_{n+\alpha}^\infty(\mathbb{R}^n)\subset L^1(\mathbb{R}^n) \cap L^\infty(\mathbb{R}^n)$. Then, there is $C>0$ such that
\begin{equation}\label{3.21}
\|S_\alpha^\beta(t) u_0\|_{L_{n+\alpha}^\infty}\leq C\|u_0\|_{L_{n+\alpha}^\infty}+Ct^\beta\|u_0\|_1.
\end{equation}
\begin{equation}\label{3.22}
\|\nabla T_\alpha^\beta(t) u_0\|_{L_{n+\alpha}^\infty}\leq Ct^{-\frac{\beta}{\alpha}+\beta-1}\|u_0\|_{L^\infty_{n+\alpha}}+Ct^{2\beta-\frac{\beta}{\alpha}-1}\|u_0\|_1.
\end{equation}
\end{prop}

\begin{proof}
For $|x|^{\alpha}t^{-\beta}\geq 1$ and $1<\alpha\le 2$, \eqref{eq:Pfaraway} implies that
\begin{equation}\label{3.27}
|x|^{n+\alpha}|P(x,t)|\leq Ct^\beta.
\end{equation}
For $|x|^{\alpha}t^{-\beta} \leq 1$, we use \eqref{eq:Psmall} to obtain
\begin{equation}\label{3.28}
|x|^{n+\alpha}|P(x,t)|\leq C t^{\beta}.
\end{equation}
(For example, if $\alpha<n$, we have $
|x|^{n+\alpha}|P(x,t)|\leq C|x|^{2\alpha}t^{-\beta}\le Ct^\beta.$)

Hence, due to \eqref{3.27} and \eqref{3.28}, we have
\begin{equation}\label{weightedP}
\||x|^{n+\alpha}P(x,t)\|_{\infty}\leq Ct^\beta.
\end{equation}

Note that there exists $C>0$ such that
\begin{equation}\label{inequality}
(1+|x|)^{n+\alpha}\leq C(1+|y|)^{n+\alpha}+C|x-y|^{n+\alpha}.
\end{equation}
It then follows from \eqref{inequality} and \eqref{3.28} that
\begin{eqnarray}
\begin{aligned}\label{BOFS2}
\|S_\alpha^\beta(t)u_0\|_{L^\infty_{\alpha+n}}&=\mbox{ess}\sup_{x\in\mathbb{R}^n}|S_\alpha^\beta(t)u_0|(1+|x|)^{n+\alpha}\\
&\leq C\mbox{ess}\sup_{x\in\mathbb{R}^n}\int_{\mathbb{R}^n} P(x-y,t)u_0(y)(1+|y|)^{n+\alpha}dy\\
&+C\mbox{ess}\sup_{x\in\mathbb{R}^n}\int_{\mathbb{R}^n}P(x-y,t)|x-y|^{n+\alpha}u_0(y)dy\\
&\leq C\|u_0\|_{L^\infty_{n+\alpha}}+Ct^{\beta}\|u_0\|_1.
\end{aligned}
\end{eqnarray}

Using \eqref{eq:gradY1} and \eqref{eq:gradY2}, we similarly find
\begin{equation}\label{eq:weightedgradY}
\| |x|^{n+\alpha}\nabla Y(x,t)\|_{\infty}\leq Ct^{2\beta-\frac{\beta}{\alpha}-1}.
\end{equation}
Further, equation \eqref{eq:gradYLp} implies
\begin{equation}\label{eq:gradYL1}
\|\nabla Y(x,t)\|_1\le C t^{-\frac{\beta}{\alpha}+\beta-1}.
\end{equation}
Therefore, by  \eqref{inequality}, \eqref{eq:weightedgradY} and \eqref{eq:gradYL1}, we similarly have
\begin{eqnarray}
\begin{aligned}\label{BOFS3}
\|\nabla T_\alpha^\beta(t)u_0\|_{L^\infty_{\alpha+n}}&=\mbox{ess}\sup_{x\in\mathbb{R}^n}|(\nabla Y)*u_0|(1+|x|)^{n+\alpha}\\
&\leq C\|\nabla Y(t)\|_1\|u_0\|_{L^\infty_{n+\alpha}}+C\| |x|^{n+\alpha}\nabla Y(x,t)\|_{\infty}\|u_0\|_1\\
&\leq Ct^{-\frac{\beta}{\alpha}+\beta-1}\|u_0\|_{L^\infty_{n+\alpha}}+Ct^{2\beta-\frac{\beta}{\alpha}-1}\|u_0\|_1.
\end{aligned}
\end{eqnarray}
\end{proof}

%================================================================================================================================

\section{Existence and uniqueness of mild solutions}\label{sec:exisuniq}

In this section, based on the $L^r-L^q$ estimates in subsection \ref{sec:Lpfund}, we construct the local existence and uniqueness of mild solution to equation \eqref{generalized Keller-Segel} for initial data $\rho_0\in L^p(\mathbb{R}^n)$ and the global existence for small initial data $\rho_0\in L^{p_c}(\mathbb{R}^n)$, where $p_c=\frac{n}{\alpha+\gamma-2}$.  Based on the weighted estimates in subsection \ref{WE}, we establish the existence and uniqueness of mild solutions in the weighed space $C([0,T],L^\infty_{n+\alpha}(\mathbb{R}^n))$. Finally, we provide the proof of integrability and integral preservation for $\rho_0\in L^1(\mathbb{R}^n)\cap L^p(\mathbb{R}^n)$ and $\rho_0\in L^1(\mathbb{R}^n)\cap L^{p_c}(\mathbb{R}^n)$.

Before we make the analysis, let us perform scaling. Suppose that $u(x, t)$ satisfies the equation \eqref{generalized Keller-Segel}.
Since \eqref{generalized Keller-Segel} formally preserves mass, we consider first the mass-preserving scaling as
\[
u_{\lambda}(x, t)=\lambda^n u(\lambda x, \lambda^b t).
\]
Then,  $(-\Delta)^{\frac{\alpha}{2}}u_{\lambda}=\lambda^{n+\alpha}(-\Delta)^{\frac{\alpha}{2}}u$ while $\nabla\cdot (u_{\lambda}B(u_{\lambda}))=\lambda^{2n+2-\gamma}\nabla\cdot(uB(u))$. Clearly, if $n+\alpha>2n+2-\gamma$, or
\[
n<\alpha+\gamma-2,
\]
the diffusion is stronger and this case is referred to the sub-critical case (in terms of mass concentration or diffusion). For usual PDEs, in the subcritical case, all $L^1$ initial data will lead to global existence of solutions. For our model, since we have used \eqref{eq:B} and assumed $\gamma\in (1, n]$, there will be no sub-critical cases. If we allow $\gamma>n$ (of course \eqref{eq:B} should be replaced by $B(\rho)=\nabla (-\Delta)^{-\gamma/2}\rho$), there can be sub-critical cases. Since sub-critical case is not our focus and the paper is already too long, we leave it for future.

If $n>\alpha+\gamma-2$,
the aggregation term can be strong, and this case will be referred to the super-critical case.  For super-critical case, there is a critical $L^p$ space. To see this, let us consider another scaling which yields another solution:
\[
u^{\lambda}=\lambda^{a}u(\lambda x, \lambda^b t).
\]
We have $_0^cD^{\beta}_tu^{\lambda}=\lambda^a\lambda^{\beta b}{_0^cD^{\beta}_tu}$, $(-\Delta)^{\frac{\alpha}{2}}u^{\lambda}=\lambda^{a+\alpha}(-\Delta)^{\frac{\alpha}{2}}u$ while $\nabla\cdot (u^{\lambda}B(u^{\lambda}))=\lambda^{2a+2-\gamma}\nabla\cdot(uB(u))$. For $u^{\lambda}$ to be a solution, we then find
\[
a+\beta b=a+\alpha,~a=\alpha+\gamma-2.
\]
In other words,
\[
u^{\lambda}=\lambda^{\alpha+\gamma-2}u(\lambda x, \lambda^{\frac{\alpha}{\beta}}t)
\]
is also a solution to \eqref{generalized Keller-Segel}. Under the transformation $u\mapsto u^{\lambda}$, the $L^{\frac{n}{\alpha+\gamma-2}}$ norm is invariant. Hence, the critical index should be
\begin{gather}\label{criticalindex}
p_c=\frac{n}{\alpha+\gamma-2}.
\end{gather}
 As we will see in Section \ref{sec:blowup}, for supercrtical case and critical case ($n=\alpha+\gamma-2$), large $L^1$ data can lead to blowup behaviors. Since the assumption $n\ge 2$ and $\gamma\le n$ implies that $p_c\ge 1$, we need to impose the small initial data in $L^{p_c}$ in order to obtain the global existence (Theorem 4.2).

To prove the existence of mild solutions, we first recall the following fixed point theorem
\begin{lem}\cite[ Lemma 3.1]{BK} \label{lmm:basicexis}
 Let $(X,\|\cdot\|_X)$ be a Banach space and $H: X\times X\rightarrow X$ be a bounded bilinear form such that for all $u_1,u_2\in X$ and a constant $\eta>0$,
\begin{equation}\label{BFH}
||H(u_1,u_2)||_X\leq \eta\|u_1||_X\|u_2||_X.
\end{equation}
If $0<\epsilon<\frac{1}{4\eta}$ and $v\in X$ such that $\|v\|_X\leq\epsilon$, then the equation $u=v+H(u,u)$ has a solution in $X$ satisfying $\|u\|_X\leq 2\epsilon$. In addition, this solution is the unique one in $\overline{B}(0,2\epsilon)$.
\end{lem}

For our model, we are going to define the bilinear form to be
\begin{gather}\label{eq:BilinearH}
H(u, v)=-\int_0^t \nabla\cdot(T_{\alpha}^{\beta}(t-s)(u(s)B(v(s))))ds.
\end{gather}

\subsection{Existence in $L^p$ spaces}

In this subsection, we investigate the existence of the solution to \eqref{generalized Keller-Segel} when the initial data is in $L^p$ space.
\begin{thm}\label{thm:localexis}
Suppose $n\ge 2$, $0<\beta<1$, $1<\alpha\le 2$ and $1<\gamma\leq n$. Let $p\in (p_c,\infty)\cap [\frac{2n}{n+\gamma-1}, \frac{n}{\gamma-1})$. Then, for any $\rho_0\in L^p(\mathbb{R}^n)$, there exists $T>0$ such that the equation \eqref{generalized Keller-Segel} admits a unique mild solution in $C([0, T]; L^p(\mathbb{R}^n))$ with initial value $\rho_0$ in the sense of Definition \ref{def:mild}. Define the largest time of existence
\[
T_b=\sup\{T>0: \text{\eqref{generalized Keller-Segel} has a unique mild solution in }C([0, T]; L^p(\mathbb{R}^n))\}.
\]
Then if $T_b<\infty$, we have
\[
\limsup_{t\to T_b^-}\|\rho(\cdot, t)\|_{p}=+\infty.
\]
\end{thm}

\begin{proof}
By Proposition \ref{pro:Lpofevoloperator} and Proposition \eqref{pro:timecontinuity}, $S_{\alpha}^{\beta}(t)\rho_0\in C([0, T]; L^p(\mathbb{R}^n))$ with
\[
\|S_{\alpha}^{\beta}(t)\rho_0\|_{C([0, T]; L^p(\mathbb{R}^n))}\le \|\rho_0\|_p.
\]
We now consider the second term in \eqref{eq:mildreformulation}. By Proposition \ref{pro:Lpofevoloperator}, we have
\begin{gather}\label{eq:Hlpest}
\begin{split}
\|H(u, v)\|_p &\le \int_0^t \|\nabla\cdot( T_{\alpha}^{\beta}(t-s) u(s)B(v(s)))\|_p ds \\
&\le C\int_0^t(t-s)^{-\frac{n\beta}{\alpha}(\frac{1}{q}-\frac{1}{p})
-\frac{\beta}{\alpha}+\beta-1}\|u(s)B(v(s))\|_q\, ds,\\
&\le C\int_0^t(t-s)^{-\frac{n\beta}{\alpha}(\frac{1}{q}-\frac{1}{p})
-\frac{\beta}{\alpha}+\beta-1}\|u(s)\|_p\|B(v(s))\|_{\frac{pq}{p-q}}\, ds,
\end{split}
\end{gather}
provided $1\le q\le p<\theta_4$ as in Proposition \ref{pro:Lpofevoloperator}.

Choosing $q$ such that
\[
\frac{1}{p}-\frac{\gamma-1}{n}=\frac{p-q}{pq},
\]
or $\frac{1}{q}=\frac{2}{p}-\frac{\gamma-1}{n}$. If $\frac{1}{q}>\frac{1}{p}$ and $q\ge 1$, or
\[
1<\frac{2n}{n+\gamma-1}\le p<\frac{n}{\gamma-1},
\]
 then Hardy-Littlewood-Sobolev inequality implies that
\[
\|B(v)(s)\|_{\frac{pq}{p-q}}\le C\|v(s)\|_p.
\]

Clearly, if $p\ge \frac{2n}{2\alpha+\gamma-2}$, $\theta_4=\infty$. Otherwise, we need $p<\theta_4=\frac{n}{2n/p-\gamma-2\alpha+2}$.
This means that we need
\[
p>\frac{n}{2\alpha+\gamma-2},
\]
which is clearly true since $p>p_c$.

With these requirements, we find
\begin{gather}\label{eq:estimateH}
\sup_{0\le t\le T}\|H(u, v)\|_{p}
\le C\sup_{0\le t\le T}\int_0^t
(t-s)^{-\frac{n\beta}{\alpha}(\frac{1}{q}-\frac{1}{p})
-\frac{\beta}{\alpha}+\beta-1}\|u(s)\|_p\|v(s)\|_p\,ds.
\end{gather}

Since $p>p_c$  implies that $-\frac{n\beta}{\alpha}(\frac{1}{q}-\frac{1}{p})
-\frac{\beta}{\alpha}+\beta>0$,  we have
\[
\sup_{0\le t\le T}\|H(u, v)\|_{p}
\le C\|u\|_{C([0, T]; L^p(\mathbb{R}^n))}\|v\|_{C([0, T]; L^p(\mathbb{R}^n))} T^{-\frac{n\beta}{\alpha}(\frac{1}{q}-\frac{1}{p})
-\frac{\beta}{\alpha}+\beta}.
\]

Now, we need to verify that $H(u, v)\in C([0,T],L^p(\mathbb{R}^n))$. Choosing the same $q$  as above, again by H\"older inequality and Hardy-Littlewood-Sobolev inequality, we have that
\[
w(s):=u(s)B(v(s))\in C([0, T]; L^q(\mathbb{R}^n)).
\]
Now, choose $0\le t<t+\delta\le T$ for some $t>0$ and $\delta>0$. Then, we pick $\delta_1>0$ and have
\begin{multline*}
\|H(u,v)(t+\delta)-H(u,v)(t)\|_p
\le \left\|\int_{\max(0,t-\delta_1)}^{t+\delta}\nabla T_{\alpha}^{\beta}(t+\delta-s)w(s)\,ds \right\|_p\\
+\left\|\int_{\max(0,t-\delta_1)}^{t}\nabla T_{\alpha}^{\beta}(t-s)w(s)\,ds
\right\|_p
+\left\|\int_0^{\max(0, t-\delta_1)}\nabla( T_{\alpha}^{\beta}(t+\delta-s)-T_{\alpha}^{\beta}(t-s))w(s)\,ds \right\|_p.
\end{multline*}
Estimates of the first two terms are similar to the argument in \eqref{eq:Hlpest} and they are controlled by \[
C\|w\|_{C([0,T],L^q(\mathbb{R}^n))}(\delta+\delta_1)^{-\frac{n\beta}{\alpha}(\frac{1}{q}-\frac{1}{p})
-\frac{\beta}{\alpha}+\beta}.
\]
By Proposition \ref{pro:timecontinuity}, $\nabla Y\in C([\delta_1,T],L^r(\mathbb{R}^n))$ for $r\in [1,\kappa_4)$ and therefore $\nabla Y$ is uniformly continuous on $[\delta_1, T]$. Hence, the third term goes to zero as $\delta \to 0$. This verifies that $H(u, v)\in C([0,T],L^p(\mathbb{R}^n))$.

Choosing $T$ small enough, Lemma \ref{lmm:basicexis} applies and the existence follows.

Because our equation is non-Markovian so that we cannot apply the continuation technique while Lemma \ref{lmm:basicexis} only implies the uniqueness in short time. Instead, we provide another direct proof for the uniqueness. Suppose that we have two solution $\rho_1$ and $\rho_2$ on $[0, T]$.
Let $M=\max(\|\rho_1\|_{C([0, T]; L^p(\mathbb{R}^n))}, \|\rho_2\|_{C([0, T]; L^p(\mathbb{R}^n))})$. We define
\[
e(t)=\|\rho_1-\rho_2\|_{C([0, t]; L^p(\mathbb{R}^n)))}.
\]
Then, by \eqref{eq:mildreformulation}, \eqref{eq:BilinearH} and \eqref{eq:estimateH}, we have
\[
e(t)\le \sup_{0\le s\le t}\|H(\rho_1, \rho_1)(s)-H(\rho_2,\rho_2)(s)\|_p
\le MC\int_0^t(t-\tau)^{-\frac{n\beta}{\alpha}(\frac{1}{q}-\frac{1}{p})
-\frac{\beta}{\alpha}+\beta+1} e(\tau)\,d\tau.
\]
The comparison principle in \cite{fllx17} (Proposition 5) implies that $e(t)=0$.

We prove the last claim regarding $T_b$ by contradiction. Assume
$\limsup_{t\to T_b^-}\|\rho\|_{p}<\infty$.
Then $\sup_{t\in [0, T_b)}\|\rho(t)\|_{p}<\infty$. Following  the same approach as we show $H(u,v)\in C([0,\infty],L^p(\mathbb{R}^n))$, and noticing the fact that $S_{\alpha}^{\beta}\rho_0\in C([0, M], L^p(\mathbb{R}^n))$ (and thus uniformly continuous) for any $M>0$,  equation \eqref{eq:mildreformulation} indicates that for any $\epsilon>0$, there exists $\delta>0$ such that $\|\rho(t_1)-\rho(t_2)\|_{p}<\epsilon$ when $T_b-\delta<t_1<t_2<T_b$.  Hence, we can define $\rho(T_b)$ so that $\rho(\cdot)\in C([0, T_b], \mathbb{R}^n)$. Now, we consider the following equation about $\tilde{\rho}$:
\begin{multline*}
\tilde{\rho}(t)=\left( S_{\alpha}^{\beta}(t+T_b)\rho_0-\int_0^{T_b} \nabla_x\cdot(T_{\alpha}^{\beta}(T_b+t-s)\rho(s)B(\rho)(s))\,ds\right)\\
-\int_0^{t} \nabla_x\cdot(T_{\alpha}^{\beta}(t-s)\tilde{\rho}(s)B(\tilde{\rho})(s))\,ds.
\end{multline*}
The first term is in $C([0, \infty), L^p(\mathbb{R}^n))$  following  the same approach as we show $H(u,v)\in C([0,\infty],L^p(\mathbb{R}^n))$. Repeating what has been just done, this new integral equation has a unique solution in $C( [0, \delta_1]; L^p(\mathbb{R}^n))$ for some $\delta_1>0$. If we define $\rho(T_b+t)=\tilde{\rho}(t)$ for $t\in [0, \delta_1]$, then $\rho$ becomes a mild solution on $[0, T_b+\delta_1)$, which contradicts with the definition of $T_b$.
\end{proof}

By the $L^r-L^q$ estimate of $S_{\alpha}^{\beta}(t)u$ in Proposition \ref{pro:Lpofevoloperator}, we find that for $p_1\in [p, \theta_1)$
\begin{gather}
\sup_{0\le t\le T}(\|S_{\alpha}^{\beta}(t)u\|_p+t^{\frac{n\beta}{\alpha}(\frac{1}{p}-\frac{1}{p_1})}\|S_{\alpha}^{\beta}(t)u\|_{p_1}) \le C\|u\|_p.
\end{gather}
This motivates us to define the following norm for $u\in C([0, T]; L^p(\mathbb{R}^n))$:
\begin{gather}\label{eq:norm1}
\|u\|_{p,p_1;T}:=\sup_{0\le t\le T}(\|u\|_p+t^{\frac{n\beta}{\alpha}(\frac{1}{p}-\frac{1}{p_1})}\|u\|_{p_1}) \le C\|u\|_p.
\end{gather}
For a given initial data in $\rho_0\in L^p(\mathbb{R}^n)$, we may use the modified norm $\|\cdot\|_{p,p_1}$ above for the function space. This is beneficial because we do not have to use $\|v\|_p$ to control $B(\rho)$ when applying the Hardy-Littlewood-Sobolev
inequality. Instead, we can use $\|\cdot\|_{p_1}$ norm with some time factor to control $B(\rho)$. Indeed, by some preliminary calculation, we find
\begin{gather*}
\begin{split}
\|H(u, v)\|_p &\le \int_0^t \|\nabla T_{\alpha}^{\beta}(t-s) u(s)B(v)(s)\|_p ds \\
&\le C\int_0^t(t-s)^{-\frac{n\beta}{\alpha}(\frac{1}{q}-\frac{1}{p})
-\frac{\beta}{\alpha}+\beta-1}\|u(s)B(v)(s)\|_q\, ds,\\
&\le C\int_0^t(t-s)^{-\frac{n\beta}{\alpha}(\frac{1}{p_1}-\frac{\gamma-1}{n})
-\frac{\beta}{\alpha}+\beta-1}\|u(s)\|_p\|v\|_{p_1}\, ds\\
&\le C\|u\|_{p,p_1}\|v\|_{p,p_1}t^{\beta(\frac{-n}{\alpha p}-\frac{2}{\alpha}+ \frac{\gamma}{\alpha}+1)}.
\end{split}
\end{gather*}
In the above estimate, $q$ and $p_1$ are related by $\frac{1}{q}=\frac{1}{p}+\frac{1}{p_1}-\frac{\gamma-1}{n}$.
According to this formula, we find that for $p>p_c$, this viewpoint provides nothing new compared with Theorem \ref{thm:localexis}.
However, we are  allowed to choose $p=p_c$ while keeping $p_1>p_c$.  This observation yields the global existence and uniqueness of the mild solution to \eqref{generalized Keller-Segel} with small data in $L^{p_c}(\mathbb{R}^n)$:
\begin{thm}\label{thm:globalexis}
Suppose $n\ge 2$, $0<\beta<1$, $1<\alpha\le 2$ and $1<\gamma\leq n$. Let $\nu=\infty$ if $2(\alpha+\gamma-2)\beta-\alpha\le 0$ or $\nu=\frac{2n\beta}{2(\alpha+\gamma-2)\beta-\alpha}$ if $2(\alpha+\gamma-2)\beta-\alpha> 0$.
Then, whenever $(p_c,\nu)\cap [\frac{n}{n-\alpha+1}, \frac{n}{\gamma-1})$ is nonempty, for any $p\in (p_c,\nu)\cap [\frac{n}{n-\alpha+1}, \frac{n}{\gamma-1})$, there exists $\delta>0$ such that for all $\rho_0\in L^{p_c}(\mathbb{R}^n)$ with
\[
\|\rho_0\|_{p_c}\le \delta,
\]
the equation \eqref{generalized Keller-Segel} admits a mild solution $\rho\in C([0, \infty); L^{p_c}(\mathbb{R}^n))$ with initial value $\rho_0$ in the sense of Definition \ref{def:mild}, satisfying
\begin{gather}
\|\rho(t)\|_{p_c}\le 2\delta, \forall t>0,
\end{gather}
and $\rho\in C((0,\infty), L^p(\mathbb{R}^n))$.
Further, the solution is unique in
\begin{gather}
X_T:=\Big\{\rho\in C([0, T]; L^{p_c}(\mathbb{R}^n))\cap C((0,T], L^p(\mathbb{R}^n)): \|\rho\|_{p_c,p; T}<\infty \Big\},~T\in (0, \infty).
\end{gather}
\end{thm}

\begin{proof}
Fix $T\in (0, \infty]$. Consider the space $X:=X_T$ with the norm $\|\cdot\|_X:=\|\cdot\|_{p_c, p; T}$. This space is then a Banach space.

By proposition \ref{pro:timecontinuity}, $S_{\alpha}^{\beta}(t)\rho_0\in C([0, T],L^{p_c}(\mathbb{R}^n))$ for any $T>0$. Since $P\in C((0,\infty),L^r(\mathbb{R}^n))$ for $r\in [1, \kappa_1)$, we then have $S_{\alpha}^{\beta}(t)\rho_0\in C((0, T],L^{p}(\mathbb{R}^n))$.
 By \eqref{eq:norm1}, we find
\[
\|S_{\alpha}^{\beta}(t)\rho_0\|_X\le C\|\rho_0\|_{p_c}\le C\delta.
\]
Hence $S_{\alpha}^{\beta}\rho_0\in X$.

For $H(u,v)$, we have:
\begin{gather}\label{eq:Hestimatespc}
\begin{split}
\|H(u, v)\|_{p_c}&\le C\int_0^t(t-s)^{-\frac{n\beta}{\alpha}(\frac{1}{q}-\frac{1}{p_c})
-\frac{\beta}{\alpha}+\beta-1}\|u(s)B(v(s))\|_q\, ds,\\
&\le C\int_0^t(t-s)^{-\frac{n\beta}{\alpha}(\frac{1}{p}-\frac{\gamma-1}{n})
-\frac{\beta}{\alpha}+\beta-1}\|u(s)\|_{p_c}\|v\|_{p}\, ds\\
&\le C\|u\|_{X}\|v\|_{X}\int_0^t(t-s)^{-\frac{n\beta}{\alpha}(\frac{1}{p}-\frac{\gamma-1}{n})
-\frac{\beta}{\alpha}+\beta-1} s^{\frac{n\beta}{\alpha}(\frac{1}{p}-\frac{1}{p_c})}\, ds\\
&\le C\|u\|_{X}\|v\|_{X},
\end{split}
\end{gather}
where H\"older's inequality and Lemma \ref{lmm:hls} imply that  $\frac{1}{q}=\frac{1}{p_c}+\frac{1}{p}-\frac{\gamma-1}{n}$.
In these inequalities, we need  $\theta_4>p_c> q\ge 1$ and $p>p_c$, and $\frac{n\beta}{\alpha}(\frac{1}{p}-\frac{1}{p_c})>-1$. Note that $p_c>q\ge 1$ ensures the H\"older inequality, the Hardy-Littlewood-Sobolev inequality (Lemma \ref{lmm:hls}) to be applied:
\[
\frac{n}{n-\alpha+1}\le p< \frac{n}{\gamma-1}.
\]
Also $p>p_c$ and $\frac{n\beta}{\alpha}(\frac{1}{p}-\frac{1}{p_c})>-1$ ensure the integrals with respect to $s$ to converge. Note that $\max(\frac{n}{n-\alpha+1}, p_c)\ge \max(\frac{2n}{n+\gamma-1},p_c)$ as it should be.
Further, $\theta_4=\infty$ if $p\ge \frac{n}{\alpha}$. If $p<\frac{n}{\alpha}$,
\[
\theta_4=\frac{n}{n(1/p_c+1/p-(\gamma-1)/n)+1-2\alpha},
\]
and $1\le q<p_c<\theta_4$ is automatically true.

The other part can be estimated in the same way.
\begin{gather}\label{eq:Hestimatespc2}
\begin{split}
t^{\frac{n\beta}{\alpha}(\frac{1}{p_c}-\frac{1}{p})}\|H(u, v)\|_{p}&\le Ct^{\frac{n\beta}{\alpha}(\frac{1}{p_c}-\frac{1}{p})}\int_0^t(t-s)^{-\frac{n\beta}{\alpha}(\frac{1}{q}-\frac{1}{p})
-\frac{\beta}{\alpha}+\beta-1}\|u(s)B(v(s))\|_q\, ds,\\
&\le Ct^{\frac{n\beta}{\alpha}(\frac{1}{p_c}-\frac{1}{p})}\int_0^t
(t-s)^{-\frac{n\beta}{\alpha}(\frac{1}{p}-\frac{\gamma-1}{n})
-\frac{\beta}{\alpha}+\beta-1}\|u(s)\|_p\|v(s)\|_p\,ds,
\end{split}
\end{gather}
where $\frac{1}{q}=\frac{2}{p}-\frac{\gamma-1}{n}$. In order to make sure the above inequalities to be held, we need $p$ to satisfy the conditions in Theorem \ref{thm:localexis}. However, as we know $\max(\frac{n}{n-\alpha+1}, p_c)\ge \max(\frac{2n}{n+\gamma-1},p_c)$, these conditions are satisfied automatically.

We then find that
\begin{eqnarray*}
\begin{aligned}
&t^{\frac{n\beta}{\alpha}(\frac{1}{p_c}-\frac{1}{p})}\|H(u, v)\|_{p}\\
&\le C\|u(s)\|_X\|v(s)\|_X t^{\frac{n\beta}{\alpha}(\frac{1}{p_c}-\frac{1}{p})}\int_0^t
(t-s)^{-\frac{n\beta}{\alpha}(\frac{1}{p}-\frac{\gamma-1}{n})
-\frac{\beta}{\alpha}+\beta-1}s^{2\frac{n\beta}{\alpha}(\frac{1}{p}-\frac{1}{p_c})}\,ds,\\
&=C\|u(s)\|_X\|v(s)\|_X.
\end{aligned}
\end{eqnarray*}
For this integral to converge, we need $2\frac{n\beta}{\alpha}(\frac{1}{p}-\frac{1}{p_c})>-1$, or $p<\nu$.

The claim that $H(u, v)\in C([0,T],L^{p_c}(\mathbb{R}^n))\cap C((0,T],L^p(\mathbb{R}^n))$ can be proved with the same argument as the proof of Theorem \ref{thm:localexis}. The only difference is the estimates of \eqref{eq:Hestimatespc} and \eqref{eq:Hestimatespc2}.
We omit the details and then $H(u,v)\in X$.

Applying Lemma \ref{lmm:basicexis}, the existence part then follows. The uniqueness can be verified in the same way as what we did in the proof of Theorem \ref{thm:localexis}.
\end{proof}

\begin{rem}
For the standard Keller-Segel equation ($\alpha=\gamma=2$ and $\beta\to 1$),  $(p_c,\nu)\cap [\frac{n}{n-\alpha+1}, \frac{n}{\gamma-1})$ is nonempty.
\end{rem}

\begin{rem}
 For $n\ge 2$, $\gamma\le n$ and thus $p_c\ge 1$,  \eqref{generalized Keller-Segel} is critical or super-critical. According to the blowup results in Section \ref{sec:blowup}, the smallness of initial data is necessary for global existence. As for usual Keller-Segel model, the global existence for all initial data holds for sub-critical case (see \cite[Section 3.2]{bellomo2015} when $n=1$). We believe \eqref{generalized Keller-Segel} (with $B(\rho)=\nabla((-\Delta)^{-\frac{\gamma}{2}}\rho)$ allowing $\gamma>n$) also has global solutions in the sub-critical cases without the assumption of smallness.
\end{rem}

\subsection{Integrability and integral preservation}

In this subsection, we explore the integrability of the mild solutions if $\rho_0\in L^1(\mathbb{R}^n)\cap L^p(\mathbb{R}^n)$. By interpolation, we know that $\rho_0\in L^{r}(\mathbb{R}^n)$ for all $r\in [1, p]$. Corresponding to Theorem \ref{thm:localexis}, we have the following theorem:
\begin{thm}\label{thm:LpL1}
Suppose $n\ge 2$, $0<\beta<1$, $1<\alpha\le 2$ and $1<\gamma\leq n$. Suppose $\rho_0\in L^1(\mathbb{R}^n)\cap L^p(\mathbb{R}^n)$ where $p\in (p_c,\infty)\cap [\frac{2n}{n+\gamma-1}, \infty)$. Then

\begin{enumerate}[(i)]

\item There exists $T>0$ such that equation \eqref{generalized Keller-Segel} admits a unique mild solution in $C([0, T]; L^1(\mathbb{R}^n))\cap C([0, T]; L^p(\mathbb{R}^n))$ with initial value $\rho_0$ in the sense of Definition \ref{def:mild}. Further, the integral is preserved, that is,
\begin{gather}
\int_{\mathbb{R}^n}\rho(x, t)\,dx=\int_{\mathbb{R}^n}\rho_0(x)\,dx.
\end{gather}

\item Define the largest time of existence
\[
T_b=\sup\{T>0: \text{\eqref{generalized Keller-Segel} has a unique mild solution in }C([0, T]; L^1(\mathbb{R}^n))\cap C([0, T]; L^p(\mathbb{R}^n)) \}.
\]
If $T_b<\infty$, we then have
\[
\limsup_{t\to T_b^-}(\|\rho(\cdot, t)\|_{1}+\|\rho(\cdot, t)\|_{p})=+\infty.
\]

\item If $p$ falls into the range in Theorem \ref{thm:localexis}, then the largest existence times for the mild solution in Theorem \ref{thm:localexis} and the mild solution here are the same.
\end{enumerate}
\end{thm}

\begin{proof}
In this proof, we consider the space
\[
X=C([0, T]; L^1(\mathbb{R}^n))\cap C([0, T]; L^p(\mathbb{R}^n)),
\]
with the norm given by
\[
\|u\|_X:=\sup_{0\le t\le T}(\|u\|_1+\|u\|_p).
\]
Then, $X$ is a Banach space.

It is clear that for any $r\in [1, p]$, $\|u\|_{C([0, T]; L^r(\mathbb{R}^n))}\le \|u\|_X$.

(i). By Proposition \ref{pro:timecontinuity}, $S_{\alpha}^{\beta}\rho_0\in X$. Thus, by Proposition \ref{pro:Lpofevoloperator}, we find that
\[
\|S_{\alpha}^{\beta}(t)\rho_0\|_X\le \|\rho_0\|_1+\|\rho_0\|_p.
\]

We then evaluate that for any $0\le t\le T$:
\begin{gather}\label{eq:L1LpH}
\begin{split}
\|H(u, v)\|_{1} &\le C\int_0^t(t-s)^{-\frac{\beta}{\alpha}+\beta-1}\|u(s)B(v)(s)\|_1\, ds\\
&\le C\int_0^t(t-s)^{-\frac{\beta}{\alpha}+\beta-1}\|u(s)\|_{p_1}\|B(v)(s)\|_{\frac{p_1}{p_1-1}}\, ds\\
&\le C\int_0^t(t-s)^{-\frac{\beta}{\alpha}+\beta-1}\|u(s)\|_{p_1}\|v\|_{p_2}\, ds\\
&\le C\|u\|_X\|v\|_XT^{-\frac{\beta}{\alpha}+\beta},
\end{split}
\end{gather}
where $1-\frac{1}{p_1}=\frac{1}{p_2}-\frac{\gamma-1}{n}$. Clearly, as long as $1+\frac{\gamma-1}{n}\in [\frac{2}{p}, 2)\cap (\frac{1}{p}+\frac{\gamma-1}{n}, 2)$, there always exist $p_1\in (1, p]$, $p_2\in (1, p]\cap (1, \frac{n}{\gamma-1})$ such that the H\"older and the Hardy-Littlewood-Sobolev  inequalities hold.
This condition is satisfied if
\[
p\ge \frac{2n}{n+\gamma-1}.
\]
The other constraint $\theta_4>1$ in Proposition \ref{pro:Lpofevoloperator} is automatically satisfied.

Similarly, we can estimate for $0\le t\le T$:
\begin{gather}\label{eq:L1LpH2}
\begin{split}
\|H(u, v)\|_{p} &\le C\int_0^t(t-s)^{-\frac{n\beta}{\alpha}(\frac{1}{q}-\frac{1}{p})
-\frac{\beta}{\alpha}+\beta-1}\|u(s)B(v(s))\|_q\, ds,\\
&\le C\int_0^t(t-s)^{-\frac{n\beta}{\alpha}(\frac{1}{q}-\frac{1}{p})
-\frac{\beta}{\alpha}+\beta-1}\|u(s)\|_{p_3}\|B(v(s))\|_{\frac{p_3q}{p_3-q}}\, ds\\
& \le C\int_0^t(t-s)^{-\frac{n\beta}{\alpha}(\frac{1}{q}-\frac{1}{p})
-\frac{\beta}{\alpha}+\beta-1}\|u(s)\|_{p_3}\|v\|_{p_4}\, ds.
\end{split}
\end{gather}
Here, we require $\theta_4>p\ge q\ge 1$, and $\frac{1}{q}-\frac{1}{p_3}=\frac{1}{p_4}-\frac{\gamma-1}{n}$.
Clearly, as long as $\frac{1}{q}+\frac{\gamma-1}{n}\in [\frac{2}{p}, 1+\frac{1}{q})\cap (\frac{1}{p}+\frac{\gamma-1}{n}, 1+\frac{1}{q})$, there always exist $p_3\in (q, p], p_4\in (1, \frac{n}{\gamma-1})\cap (1, p]$ to make the H\"older and the Hardy-Littlewood-Sobolev inequalities hold. Hence, we need the following condition
\[
\frac{1}{q}\in [\frac{1}{p}, 1]\cap [\frac{2}{p}-\frac{\gamma-1}{n}, 2-\frac{\gamma-1}{n})\cap (\frac{1}{p}, 2-\frac{\gamma-1}{n})
=(\frac{1}{p},1 ]\cap [\frac{2}{p}-\frac{\gamma-1}{n}, 1].
\]
The interval is nonempty since $p\ge \frac{2n}{n+\gamma-1}$.
To make the integrals converge, we need
\[
\frac{1}{q}<\frac{1}{p}+\frac{\alpha-1}{n}.
\]
Therefore, we need $\frac{2}{p}-\frac{\gamma-1}{n}<\frac{1}{p}+\frac{\alpha-1}{n}$ to ensure $q$ to exist and the inequality is satisfied since $p>p_c$. $p<\theta_4$ requires $\frac{1}{q}<\frac{1}{p}+\frac{2\alpha-1}{n}$ which is guaranteed by $p>p_c$ also.

The claim that $H(u, v)\in C([0,T],L^{1}(\mathbb{R}^n))\cap C([0, T],L^p(\mathbb{R}^n))$ can be proved as in the proof of Theorem \ref{thm:localexis}. The only difference is that we use estimates as in \eqref{eq:L1LpH} and \eqref{eq:L1LpH2}.
We omit the details and then $H(u,v)\in X$.

Then, we find that there exists some $\delta>0$ such that
\[
\|H(u, v)\|_X\le CT^{\delta}\|u\|_X\|v\|_X.
\]
The existence for small time then follows from Lemma \ref{lmm:basicexis}.  The uniqueness argument is performed similarly as we did in Theorem \ref{thm:localexis}.

(ii).  The statement regarding $T_b$ follows in a similar approach as the proof in Theorem \ref{thm:localexis} and we omit the details.

(iii).  Consider that $p$ is in the range as in Theorem \ref{thm:localexis}. To prove that the largest existence times are the same, we only have to show that if $\rho\in C([0, T_1]; L^p(\mathbb{R}^n))$ is a mild solution in Theorem \ref{thm:localexis} for some $T_1>0$, then $\rho\in C([0, T_1]; L^1(\mathbb{R}^n))$. To this end, we consider the sequence $\{\rho^n\}$ defined inductively by $\rho^0=S_{\alpha}^{\beta}\rho_0$ and
\[
\rho^{n+1}=S_{\alpha}^{\beta}\rho_0+H(\rho, \rho^n), ~n\ge 0.
\]
By the same estimates as in Theorem \ref{thm:localexis}, we find that
$\rho^n\in C([0, T], L^p(\mathbb{R}^n))$. Direct computation shows that $\rho^n\to \rho$ in $C([0, T_1]; L^p(\mathbb{R}^n))$ as $n\rightarrow \infty$.
Now, by \eqref{eq:L1LpH}, we find
\[
\|H(\rho, \rho^n)\|_{1}\le C\int_0^t(t-s)^{-\frac{\beta}{\alpha}+\beta-1}\|\rho(s)\|_{p}\|\rho^n(s)\|_{\frac{np}{np+p\gamma-n-p}}\, ds.
\]
Since $\frac{np}{np+p\gamma-n-p}\in (1, p]$, there is $\sigma\in [0, 1)$ such that
\begin{eqnarray*}
\begin{aligned}
\|H(\rho, \rho^n)\|_{1}&\le C\int_0^t(t-s)^{-\frac{\beta}{\alpha}+\beta-1}\|\rho^n(s)\|_1^{\sigma}\|\rho^n(s)\|_p^{1-\sigma}\, ds\\
&\le C_1\int_0^t(t-s)^{-\frac{\beta}{\alpha}+\beta-1}\|\rho^n(s)\|_1^{\sigma}\,ds.
\end{aligned}
\end{eqnarray*}

This gives the inequality
\[
\|\rho^{n+1}\|_1\le \|\rho_0\|_1+C_1(T_1)\int_0^t(t-s)^{-\frac{\beta}{\alpha}+\beta-1}\|\rho^n(s)\|_1^{\sigma}\,ds.
\]
Inductively, it is easy to see that for any positive integer $m$,  $\|\rho^m(t)\|_1\le u(t)$, where $u(t)$ solves the fractional ODE $D_c^{\beta-\frac{\beta}{\alpha}}u=C_1(T_1)u^{\sigma}$,$u(0)=\|\rho_0\|_1$, which exists globally as showed in \cite{fllx17}.
Extracting an almost everywhere convergent subsequence and applying Fatou's lemma to $|\rho^{n_k}|$ and $|\rho|$, we find $\|\rho\|_1\le u(t)$. Using this boundedness and similar estimate to the proof in equation \eqref{eq:L1LpH}, it is straightforward to show $\|\rho(t+\Delta t)-\rho(t)\|_1\to 0$ as $\Delta t \to 0$. Hence, $\rho\in C([0, T_1]; L^1(\mathbb{R}^n))$.

As soon as we have the integrability, we have
\begin{equation}\label{4.1}
\int_{\mathbb{R}^n}\rho\,dx=\int_{\mathbb{R}^n} S_{\alpha}^{\beta}\rho_0\, dx-\int_{\mathbb{R}^n}\int_0^t\nabla\cdot(T_{\alpha}^{\beta}(t-s)\rho(s) B(\rho(s)))ds\,dx.
\end{equation}
Since we know $P\ge 0$ and $\int_{\mathbb{R}^n} P\,dx=1$, we then have
\begin{equation}\label{4.2}
\int_{\mathbb{R}^n} S_{\alpha}^{\beta}\rho_0\, dx=\int_{\mathbb{R}^n} \rho_0\,dx.
\end{equation}
Further, for any $t>0$, we have $\zeta=\rho(s) B(\rho)(s)\in C([0, t]; L^1({\mathbb{R}^n}))$ since $\rho\in X$. By approximating $\zeta$ with $C_c^{\infty}([0, t]\times \mathbb{R}^n)$ functions, we find
\begin{equation}\label{4.3}
\int_{\mathbb{R}^n}\int_0^t\nabla\cdot(T_{\alpha}^{\beta}(t-s)\zeta(s))ds\,dx=0.
\end{equation}
\eqref{4.1}, \eqref{4.2} and \eqref{4.3} prove that $L^1$ integral is preserved.
\end{proof}

\begin{rem}
In Theorem \ref{thm:LpL1}, we do not need to ask for $p<\frac{n}{\gamma-1}$ compared with Theorem \ref{thm:localexis}.  The reason is that we only choose $p_4<\frac{n}{\gamma-1}$, and $p_3$ can be adjusted accordingly, which will not affect $p$.
\end{rem}

For $\rho_0\in L^1(\mathbb{R}^n)\cap L^{p_c}(\mathbb{R}^n)$, we have the following claim:
\begin{thm}\label{thm:L1critical}
Assume the conditions and notations in Theorem \ref{thm:globalexis} hold. Suppose $(p_c,\nu)\cap [\frac{n}{n-\alpha+1}, \frac{n}{\gamma-1})$ is nonempty. Then, there exists $\delta>0$ such that for all $\rho_0\in L^1(\mathbb{R}^n)\cap L^{p_c}(\mathbb{R}^n)$ with $\|\rho_0\|_{p_c}\le\delta$, all the claims for the mild solution $\rho$  in Theorem \ref{thm:globalexis} hold and further $\rho\in C([0,\infty), L^1(\mathbb{R}^n))$ satisfying
\begin{gather}\label{4.9}
\int_{\mathbb{R}^n}\rho(x, t)\,dx=\int_{\mathbb{R}^n}\rho_0(x)\,dx.
\end{gather}
\end{thm}

\begin{proof}
As in the proof of Theorem \ref{thm:globalexis}, we pick $p\in (p_c,\nu)\cap [\frac{n}{n-\alpha+1}, \frac{n}{\gamma-1})$ and $\delta>0$ so that the mild solution satisfies
\begin{gather}\label{eq:Xnorm1}
\|\rho\|_X=\sup_{t>0} \|\rho\|_{p_c}+t^{\frac{n\beta}{\alpha}(\frac{1}{p_c}-\frac{1}{p})}\|\rho\|_p\le 2\delta,
\end{gather}
and using \eqref{eq:Xnorm1}, it holds that
\begin{gather}\label{eq:smallH}
\begin{split}
 \|H(\rho, v)\|_{p_c} &\le
 C\int_0^t(t-s)^{-\frac{n\beta}{\alpha}(\frac{1}{q}-\frac{1}{p_c})
-\frac{\beta}{\alpha}+\beta-1}\|\rho(s)B(v(s))\|_q\, ds\\
&\le C\int_0^t(t-s)^{-\frac{n\beta}{\alpha}(\frac{1}{p}-\frac{\gamma-1}{n})
-\frac{\beta}{\alpha}+\beta-1}\|\rho(s)\|_{p}\|v\|_{p_c}\, ds\\
& \le C_1\|\rho\|_X\|v\|_{C([0, t]; L^{p_c})}
\le \frac{1}{2}\|v\|_{C([0, t]; L^{p_c})},
\end{split}
\end{gather}
where the same constraint $\frac{1}{q}=\frac{1}{p_c}+\frac{1}{p}-\frac{\gamma-1}{n}$ holds. The only difference is that we applied $L^{p_c}$ norm on $v$ here while we applied $L^{p_c}$ norm on the first function in the proof of Theorem \ref{thm:globalexis}. Note that $\delta$ is chosen such that $C_12\delta\le\frac{1}{2}$ which is the same as what we did in the proof of Theorem \ref{thm:globalexis}.

Clearly,
\begin{equation}\label{4.12}
 \|S_{\alpha}^{\beta}(t)\rho_0\|_1\le \|\rho_0\|_1.
 \end{equation}

If $p_1=\frac{n}{n-\alpha+1}\ge p_c$ or $n\le 2\alpha+\gamma-3$,
\begin{gather}\label{4.13}
\begin{split}
\|H(\rho, \rho)\|_{1} &\le C\int_0^t(t-s)^{-\frac{\beta}{\alpha}+\beta-1}\|\rho(s)\|_{p_c}\|B(\rho(s))\|_{\frac{p_c}{p_c-1}}\, ds,\\
&\le C\int_0^t(t-s)^{-\frac{\beta}{\alpha}+\beta-1}\|\rho(s)\|_{p_c}\|\rho(s)\|_{\frac{n}{n-\alpha+1}}\, ds\\
& \le C \|\rho\|_{X}^2
\int_0^t(t-s)^{-\frac{\beta}{\alpha}+\beta-1} s^{\frac{n\beta}{\alpha}(\frac{n-\alpha+1}{n}-\frac{1}{p_c})}\,ds,
\end{split}
\end{gather}
where the norm $X$ is in \eqref{eq:Xnorm1}. It is clear that
\begin{equation}\label{4.14}
\frac{n\beta}{\alpha}(\frac{n-\alpha+1}{n}-\frac{1}{p_c})=\frac{\beta}{\alpha}(n-2\alpha-\gamma+3)\ge \frac{\beta}{\alpha}(3-2\alpha)>-1.
\end{equation}
The integral in the third inequality of \eqref{4.13} converges. Hence, \eqref{4.12} and \eqref{4.13} imply that $\rho(t)\in L^1(\mathbb{R}^n)$. It is then easy to show that
\[
\|\rho(t+\Delta t)-\rho(t)\|_1\to 0,~\Delta t\to 0,~\forall t\ge 0,
\]
by using similar controls. Approximating $\rho B(\rho)$ with $C_c^{\infty}([0, t]\times\mathbb{R}^n)$ functions, we can prove easily that
\[
\int_{\mathbb{R}^n} H(\rho, \rho)\,dx=0.
\]
This then proves the claims when $p_1\ge p_c$.

Now, we assume $p_1<p_c$ ($p_1>1$ is clearly true). Then, using interpolation, there exists $\sigma\in (0, 1)$ such that
\begin{equation}\label{4.17}
\|u\|_{p_1}\le \|u\|_1^{\sigma}\|u\|_{p_c}^{1-\sigma}.
\end{equation}
We construct the sequence $\rho^n$ as $\rho^0=S_{\alpha}^{\beta}\rho_0$ and
\begin{equation}\label{4.18}
\rho^{n+1}=S_{\alpha}^{\beta}\rho_0+H(\rho, \rho^n), ~n\ge 0
\end{equation}
where $\rho$ is the mild solution. Inductively, with the help of \eqref{4.18}, \eqref{eq:smallH} and \eqref{eq:Xnorm1}, assumption $\|\rho_0\|_{p_c}\leq\delta$ implies that
\[
\|\rho^{n+1}\|_{p_c}\le \delta+\|H(\rho, \rho^n)\|_{p_c}\le
\delta+\frac{1}{2}2\delta=2\delta.
\]
Similarly, $\rho^{n+1}\in C([0, \infty); L^{p_c}(\mathbb{R}^n))$ can be verified directly.

Further, let $u(t)$ be the solution to the following equation
\[
_0^cD_t^{\beta-\frac{\beta}{\alpha}}u=C(2\delta)^{2-\sigma} u^{\sigma},~u(0)=\|\rho_0\|_1,
\]
which exists globally and is increasing on $[0,\infty)$ by the result in \cite{fllx17}. Clearly, $\|\rho^0\|_1\le \|\rho_0\|_1\le u(t)$.  Then, we prove by  induction that for any positive integer $m$, $\|\rho^m(t)\|_1\le u(t)$. Indeed, using \eqref{eq:Xnorm1}, \eqref{4.17} and induction assumption, we have
\begin{gather*}
\begin{split}
\|\rho^{n+1}\|_1 &\le \|\rho_0\|_1+C\int_0^t(t-s)^{-\frac{\beta}{\alpha}+\beta-1}\|\rho(s)\|_{p_c}\|\rho^n(s)\|_{\frac{n}{n-\alpha+1}}\, ds \\
& \le \|\rho_0\|_1+C(2\delta)^{2-\sigma}\int_0^t(t-s)^{-\frac{\beta}{\alpha}+\beta-1} \|\rho^n(s)\|_1^{\sigma}ds \\
 &\le \|\rho_0\|_1+C(2\delta)^{2-\sigma}\int_0^t(t-s)^{-\frac{\beta}{\alpha}+\beta-1}u^{\sigma}(s)ds= u(t).
\end{split}
\end{gather*}
Again by estimating directly $\|\rho^{n+1}(t+\Delta t)-\rho^{n+1}(t)\|_1$, we find that $\rho^{n+1}\in C([0,\infty), L^1(\mathbb{R}^n))$.

Finally, we have
\[
\rho^{n+1}-\rho=H(\rho,\rho^n-\rho)
\]
and therefore
\[
\|\rho^{n+1}-\rho\|_{p_c}\le \frac{1}{2}\|\rho^n-\rho\|_{p_c}
\]
according to \eqref{eq:smallH}. This implies that $\rho^n\to \rho$ in $C([0,\infty), L^{p_c}(\mathbb{R}^n))$. By a diagonal argument, we can take a subsequence $\rho^{n_k}\to \rho$
a.e. on $[0, \infty)\times \mathbb{R}^n$. Applying Fatou's Lemma to the absolute value of $\rho^{n_k}$ and $\rho$, we have
$\rho(t)\in L^1(\mathbb{R}^n)$ and $\|\rho\|_1\le u(t)$ for a.e. $t\in [0, \infty)$.
Using $\|\rho\|_1\le u(t)$, we can then check directly that
\[
\|\rho(t+\Delta t)-\rho(t)\|_1\to 0,~\Delta t\to 0.
\]
Again, we can prove
\[
\int_{\mathbb{R}^n} H(\rho, \rho)\,dx=0
\]
by approximating $\rho$ with smooth functions.
\end{proof}

\subsection{Existence in the weighted space}

In this subsection, we study the existence of the mild solutions to \eqref{generalized Keller-Segel}  in the weighted spaces (see equation \eqref{WLI} for the definition of $L_{\nu}^{\infty}$)

\begin{gather}
X_T=L^{\infty}([0,T],L^\infty_{n+\alpha}(\mathbb{R}^n)).
\end{gather}
For mild solutions in weighted spaces of some PDEs, one may refer to  \cite{brandolese2008,BK}

It is easy to verify that
\begin{equation}\label{w412}
\|u\|_{L_{n+\alpha}^{\infty}}<\infty  \Rightarrow  u\in L^1(\mathbb{R}^n)\cap L^{\infty}(\mathbb{R}^n).
\end{equation}

First, we establish the $L^\infty$ estimate of linear operator $B$ using the $X_T$-norm.
\begin{lem}\label{lmm:inftyB}
Let $1<\gamma\leq n$, $1<\alpha\leq2$, $n\geq2$. Assume $u\in X_T$. $B(u)$ is defined as (1.3). Then
\begin{equation}\label{B1}
\|B(u)\|_\infty\leq C\|u\|_{X_T}.
\end{equation}
\end{lem}
\begin{proof}
According to \eqref{eq:B}, we have
\begin{eqnarray*}
\begin{aligned}
|B(u)(x,t)|&=|\int_{\mathbb{R}^n}\frac{-s_{n,\gamma}(x-y)}{|x-y|^{n-\gamma+2}}u(y,t)\,dy|\\
&\le C\|u\|_{X_T}\int_{\mathbb{R}^n}\frac{1}{|x-y|^{n-\gamma+1}(1+|y|^{n+\alpha})}\,dy.
\end{aligned}
\end{eqnarray*}
Denote $\Omega_1=\{y\in\mathbb{R}^n:|y|\leq\frac{|x|}{2}\}$, $\Omega_2=\{y\in\mathbb{R}^n:|x-y|\leq\frac{|x|}{2}\}$ and $\Omega_3=\{y\in\mathbb{R}^n:|x-y|\geq\frac{|x|}{2}, |y|\geq\frac{|x|}{2}\}$.
Then, we have
\begin{gather}
I=\int_{\mathbb{R}^n}\frac{1}{|x-y|^{n-\gamma+1}(1+|y|^{n+\alpha})}\,dy
=\int_{\Omega_1}+\int_{\Omega_2}+\int_{\Omega_3}=:I_1+I_2+I_3.
\end{gather}

If $y\in \Omega_1$, we have $\frac{1}{2}|x|\le |x-y|\le \frac{3}{2}|x|$. This inequality gives that
\[
I_1\le C\frac{1}{|x|^{n-\gamma+1}}\int_{|y|\le |x|/2}\frac{1}{(1+|y|)^{n+\alpha}}\,dy
=C_1\frac{1}{|x|^{n-\gamma+1}}\int_0^{|x|/2}\frac{r^{n-1}}{(1+r)^{n+\alpha}}dr.
\]
When $|x|\le 1$, this is controlled by
\[
C_1\frac{1}{|x|^{n-\gamma+1}}\int_0^{|x|/2}r^{n-1}dr=C_2|x|^{\gamma-1}\le C_2.
\]
When $|x|>1$, we have trivially $I_1\le C_1\int_0^{\infty}\frac{r^{n-1}}{(1+r)^{n+\alpha}}dr<\infty$.

If $y\in\Omega_2$, we have $|y|\ge |x|-|x-y|\ge \frac{1}{2}|x|$.
Then,
\begin{gather}
\begin{split}
|I_2| &\le C\frac{1}{(1+|x|)^{n+\alpha}}\int_{\Omega_2}\frac{1}{|x-y|^{n-\gamma+1}}\,dy \\
&=C_1\frac{1}{(1+|x|)^{n+\alpha}}\int_0^{|x|/2}\frac{r^{n-1}}{r^{n-\gamma+1}}\,dr \\
&=C_2\frac{|x|^{\gamma-1}}{(1+|x|)^{n+\alpha}}\le C_3.
\end{split}
\end{gather}

If $y\in\Omega_3$, one has $\frac{|x|}{2}\leq |x-y|\leq 3|y|$. Then,
\begin{gather*}
|I_3|\le \int_{|x-y|\le 1}\frac{1}{|x-y|^{n-\gamma+1}}\,dy
+\int_{\Omega_3\cap \{|x-y|>1\}}\frac{1}{|x-y|^{n-\gamma+1}(1+|y|^{n+\alpha})}\,dy.
\end{gather*}
The first term is trivially bounded by a constant, and for the second term we have
\[
\begin{split}
\int_{\Omega_3\cap \{|x-y|>1\}}\frac{1}{|x-y|^{n-\gamma+1}(1+|y|^{n+\alpha})}\,dy
 &\le C_1\int_{\Omega_3\cap \{|x-y|>1\}}\frac{1}{|x-y|^{2n+\alpha-\gamma+1}}\,dy\\
 &\le C_2\int_{\max(1, |x|/2)}^{\infty}\frac{r^{n-1}}{r^{2n+\alpha-\gamma+1}}dr\\
 &\le C_3.
\end{split}
\]
All the bounds are independent of $x$ and the claim is therefore proved.
\end{proof}

Using Lemma \ref{lmm:inftyB}, we establish the existence of mild solution to \eqref{generalized Keller-Segel} in $X_T$.
\begin{thm} \label{thm:weighted}
For $n\geq 2$. Let $0<\beta<1, 1<\alpha \le 2$ and $1<\gamma\leq n$. If $\rho_0\in L^\infty_{n+\alpha}(\mathbb{R}^n)$. Then there exists $T>0$ such that \eqref{generalized Keller-Segel} has a unique mild solution $\rho\in L^{\infty}([0,T];L^\infty_{n+\alpha}(\mathbb{R}^n))$, satisfying
\[
\int_{\mathbb{R}^n}\rho\,dx=\int_{\mathbb{R}^n}\rho_0\,dx.
\]
Define the largest time of existence
\[
T_b^{\alpha}=\sup\{T>0: \text{\eqref{generalized Keller-Segel} has a unique mild solution in }X_T \}.
\]
If $T_b^{\alpha}<\infty$, we then have $\limsup_{t\to T_b^-}\|\rho(\cdot, t)\|_{L_{n+\alpha}^{\infty}}=+\infty$.
Further,  this solution is the same solution as in Theorems \ref{thm:LpL1} on $[0, T_b^{\alpha})$, and
\[
\rho\in C([0, T_b^{\alpha}), L^p(\mathbb{R}^n)),~\forall p\in [1,\infty).
\]
\end{thm}
\begin{proof}
We can construct the existence of solution to \eqref{generalized Keller-Segel} in $X_T$ applying Proposition \ref{pro:weightedfund}.

By equation \eqref{3.21} and \eqref{w412}, we find
\begin{equation}\label{w413}
\|S_\alpha^\beta(\cdot)\rho_0\|_{X_T}\le C(1+T^{\beta})\|\rho_0\|_{L_{n+\alpha}^{\infty}}.
\end{equation}

Due to the definition of the bilinear form \eqref{eq:BilinearH} and estimate \eqref{3.22}, we have
\begin{eqnarray}
\begin{aligned}\label{Bilinearform1}
\|H(\rho,\tilde{\rho})\|_{X_T}&=\mbox{ess}\sup_{x\in\mathbb{R}^n}(1+|x|)^{n+\alpha}\left|\int_0^tT_\alpha^\beta(t-\tau)\nabla(\rho B(\tilde{\rho}))(\tau)d\tau \right|\\
&\leq \int_0^t\|\nabla\cdot( T_\alpha^\beta(t-\tau)(\rho B(\tilde{\rho}))(\tau))\|_{L^\infty_{n+\alpha}}d\tau\\
&\leq C\int_0^t(t-\tau)^{-\frac{\beta}{\alpha}+\beta-1}\|(\rho B(\tilde{\rho}))(\tau)\|_{L^\infty_{n+\alpha}}+(t-\tau)^{2\beta-\frac{\beta}{\alpha}-1}\|(\rho B(\tilde{\rho}))(\tau)\|_1 d\tau
\end{aligned}
\end{eqnarray}

By Lemma \ref{lmm:inftyB} and \eqref{w412}, we can deduce
\begin{gather}\label{w415}
\|\rho B(\tilde{\rho})(\tau)\|_{L^\infty_{\alpha+n}}=\|(1+|x|)^{n+\alpha}\rho(\tau)B(\tilde{\rho})(\tau)\|_{\infty}
\leq C\|\rho\|_{X_T}\|\tilde{\rho}\|_{X_T}.
\end{gather}
and
\begin{gather}\label{w416}
\|(\rho B(\tilde{\rho}))(\tau)\|_1 \leq C\|\rho\|_1\| B(\tilde{\rho})(\tau)\|_\infty \leq C\|\rho\|_{X_T}\|\tilde{\rho}\|_{X_T}.
\end{gather}

It follows from \eqref{Bilinearform1}, \eqref{w415} and \eqref{w416} that
\begin{equation}\label{WHE}
\|H(\rho,\tilde{\rho})\|_{X_T}\leq C t^{\beta-\frac{\beta}{\alpha}}\|\rho\|_{X_T}\|\tilde{\rho}\|_{X_T}+C t^{2\beta-\frac{\beta}{\alpha}}\|\rho\|_{X_T}\|\tilde{\rho}\|_{X_T}.
\end{equation}

Combining \eqref{w413} and \eqref{WHE}, Lemma \ref{lmm:basicexis} yields the existence result in $X_T$ .

Since $\rho \in L^{\infty}([0, T], L^1(\mathbb{R}^n))\cap L^{\infty}([0, T], L^{\infty}(\mathbb{R}^n))$, we can deduce that $\rho \in C([0, T], L^p(\mathbb{R}^n))$ for any $p\in [1,\infty)$ using \eqref{eq:mildreformulation}. On the other hand, $\rho_0\in L^\infty_{n+\alpha}(\mathbb{R}^n)$ implies that it is in $L^p(\mathbb{R}^n)$ for any $p\in [1,\infty]$.  Then, Theorem \ref{thm:LpL1}  tells us that there is a unique mild solution $\tilde{\rho}$ in
$C([0, T_1], L^1(\mathbb{R}^n))\cap C([0, T_1], L^{p_1}(\mathbb{R}^n))$ for some $T_1>0$ and $p_1\in (1,\infty)$.
 The uniqueness part of Theorem \ref{thm:LpL1} ensures that $\rho=\tilde{\rho}$ on $[0,\min(T,T_1)]$. Moreover, equation \eqref{w412} and the blowup scenario of Theorem \ref{thm:LpL1} imply that $\tilde{\rho}$ exists on $[0, T_b^{\alpha})$. Theorem \ref{thm:LpL1} again then ensures  the uniqueness of $\rho$ and the integral preservation. Lastly, since $\rho\in C([0, T], L^{1}(\mathbb{R}^n))\cap L^{\infty}([0, T], L^{\infty}(\mathbb{R}^n))$, we deduce that $\rho\in C([0, T], L^{p}(\mathbb{R}^n))$ for all $p\in [1,\infty)$.

 The statement regarding $T_b^{\alpha}$ follows in a similar approach as the argument in Theorem \ref{thm:localexis} and we omit the details.

\end{proof}

\begin{rem}\label{rmk:moment}
\begin{itemize}
\item Note that we usually do not have $\rho\in C([0, T], L_{n+\alpha}^{\infty}(\mathbb{R}^n))$. Consider that $\rho_0$ is compactly supported but essentially discontinuous. It is well-known that the mollification of $\rho_0$ does not converge to $\rho_0$ in $L^{\infty}(\mathbb{R}^n)$. Hence, $S_{\alpha}^{\beta}(t)\rho_0 \not\to \rho_0$ in $L_{n+\alpha}^{\infty}(\mathbb{R}^n)$.

\item
The fact that $\rho\in X_T$ clearly implies that for $\nu<\alpha$
\[
\int_{\mathbb{R}^n} |x|^{\nu}|\rho|\,dx<\infty.
\]
Indeed, according to the asymptotic form of the fundamental solution (equation \eqref{eq:Pfaraway}), if $\alpha\in (1,2)$, we cannot expect the solution to have moment of order $\nu\ge \alpha$ (see also \cite[Section 2]{BK}).
\end{itemize}
\end{rem}

\section{Nonnegativity preservation and conservation of mass}\label{sec:positivity}
In this subsection, we intend to discuss the problem that if the fact that the initial value $\rho_0$ of \eqref{generalized Keller-Segel} is nonnegative can imply that $\rho(x,t)$ remains nonnegative for every $0<t<T$.

Given a function $u$, we recall $u^-=-\min(u, 0)\ge 0$, $u^+=\max(u, 0)$ so that
\[
u=u^+-u^-.
\]

We first have the following claim, which essentially follows from \cite[Theorem 1.1]{bouleau1986}.
\begin{lem}\label{lmm:contraction}
Let $A=(-\Delta)^{\frac{\alpha}{2}}$. Suppose $\rho\in L^2(\mathbb{R}^n)$, then for any $t>0$,
\begin{gather}
\langle e^{-tA}\rho, \rho^+\rangle\le \|\rho^+\|_2^2=\int_{\mathbb{R}^+}\rho \rho^+\,dx.
\end{gather}
\end{lem}
\begin{proof}
The kernel $P$ corresponding to $e^{-sA}$ is non-negative and integrates to $1$ for any $s>0$, which is the $\beta\to 1$ limit of Lemma \ref{lmm:PQ}. Then, it follows that
\begin{gather*}
\langle  e^{-tA}\rho, \rho^+\rangle
=\langle e^{-tA}\rho^+, \rho^+\rangle -\langle e^{-tA}\rho^-, \rho^+\rangle
\le \langle e^{-tA}\rho^+, \rho^+\rangle=\|e^{-\frac{tA}{2}}\rho^+\|_2^2\le \|\rho^+\|_2^2.
\end{gather*}
The first inequality follows because $e^{-tA}\rho^-$ is a non-negative function since $P\ge 0$, and the second inequality follows from the fact that for any $p\ge 1$, $e^{-\frac{tA}{2}}$ is a contraction in $L^p(\mathbb{R}^n)$ space (because $\|P\|_1=1$).
\end{proof}

\begin{cor}\label{cor:lemfrac}
Let $A=(-\Delta)^{\frac{\alpha}{2}}$. If $\rho\in H^{\delta,2}(\mathbb{R}^n)$ for $\delta\in [0, 1]$ and $A\rho\in H^{-\delta,2}(\mathbb{R}^n)$, then we have
\begin{gather}
\langle A\rho, \rho^+\rangle \ge 0,~~~\langle A\rho, \rho^-\rangle \le 0.
\end{gather}
\end{cor}
\begin{proof}
By Lemma \ref{lmm:contraction} and the fact $L^2(\mathbb{R}^n)\subset H^{\delta,2}(\mathbb{R}^n)$, we have
\begin{equation}\label{5.3}
\langle (e^{-tA}\rho-\rho), \rho^+\rangle\le 0.
\end{equation}
Hence, using \eqref{5.3}, we find
\begin{equation}\label{5.4}
\langle A\rho, \rho^+\rangle=-\left\langle \lim_{t\to 0}\frac{1}{t}(e^{-tA}\rho-\rho), \rho^+\right\rangle
\ge -\limsup_{t\to 0}\frac{1}{t}\langle e^{-tA}\rho-\rho, \rho^+\rangle \ge 0.
\end{equation}

Inequality \eqref{5.4} implies that
\[
-\langle e^{-tA}\rho, \rho^-\rangle =\langle e^{-tA}(-\rho), (-\rho)^+\rangle \ge 0,
\]
and the second claim follows.
\end{proof}

\begin{thm}\label{thm:positivity}
In Theorem \ref{thm:localexis} (or Theorem \ref{thm:globalexis}, Theorem \ref{thm:LpL1}, Theorem \ref{thm:L1critical}, Theorem \ref{thm:weighted}), if we also have  $\rho_0\ge 0$, then for all $t$ in the interval of existence we have
\begin{gather}
\rho(x, t)\ge 0.
\end{gather}
\end{thm}
\begin{proof}

We will just prove the case that the solutions satisfy the assumption in Theorem \ref{thm:localexis} here. The proof for other cases are similar.

We introduce a mollifier $J_{\epsilon}(x)=\frac{1}{\epsilon^n}J(\frac{x}{\epsilon})$ and consider the operator $B_{\epsilon}$ defined by
\[
B_{\epsilon}(u):=B(J_{\epsilon}*u)=J_{\epsilon}*B(u).
\]
Recall that $T_b$ is the largest time of existence. We fix $T\in (0, T_b)$ and let $\rho$ be the mild solution on $[0, T]$. We define $\rho(t)=\rho(T)$ for $t\ge T$.

Now, we first pick approximating sequence $\varrho_0^{(n)}\in L^p(\mathbb{R}^n)\cap L^2(\mathbb{R}^n)$ such that $\varrho_0^{(n)}\ge 0$ and $\varrho_0^{(n)}\to \rho_0$ in $L^p(\mathbb{R}^n)$. For example, we can choose $f_n\in C_c^{\infty}(\mathbb{R}^n)$ such that $f_n\to \rho_0$ in $L^p(\mathbb{R}^n)$. Denote $f_n\vee0:=\max(f_n,0)$ and picking $\varrho_0^{(n)}=f_n\vee 0$ suffices because $|\rho_0-f_n\vee 0|\le |\rho_0-f_n|$ due to the fact $\rho_0\ge 0$.

For the purpose of the proof, we consider the following auxiliary problem
\begin{gather}\label{eq:rhon}
 \varrho^{(n)}(t)=S_{\alpha}^{\beta}(t)\rho_0-\int_0^t  T_{\alpha}^{\beta}(t-s)\nabla\cdot((\varrho^{(n)})^+ B_{\frac{1}{n}}(\rho))\,ds.
\end{gather}

We denote
\[
a_n(x,t):=B_{\frac{1}{n}}(\rho),
\]
and it is  a smooth function with derivatives bounded in $[0,\infty)\times\mathbb{R}^n$ by the properties of mollification.

For $t>0$, since $E_{\beta}(-s)\sim C_1s^{-1}$ as $s\to\infty$, we observe
\begin{gather}\label{eq:estH1}
\|S_{\alpha}^{\beta}u\|_{H^{\sigma,2}}^2=\int_{\mathbb{R}^n}(1+|\xi|^{2\sigma}) (E_{\beta}(-t^{\beta}|\xi|^{\alpha}))^2|\hat{u}_{\xi}|^2d\xi\le C\|u\|_2^2t^{-\frac{2\sigma \beta}{\alpha}},~\sigma\in [0,\alpha].
\end{gather}
To deal with the second term in \eqref{eq:rhon}, due to the definition of $T_\alpha^\beta$, we note that
\[
\|\int_0^t T_{\alpha}^{\beta}(t-s)\nabla\cdot v\,ds\|
\le \int_0^t (t-s)^{\beta-1}\|E_{\beta,\beta}(-(t-s)^{\beta}A)\nabla\cdot v(s)\|\,ds
\]
holds for any norm. Hence, aiming to compute $H^{\sigma,2}$ norm, we have
\begin{eqnarray}
\begin{aligned}\label{5.5}
&\|E_{\beta,\beta}(-(t-s)^{\beta}A)\nabla\cdot v(s)\|_{H^{\sigma,2}}^2\\
=&\int_{\mathbb{R}^n}(1+|\xi|^{2\sigma}) \left|\xi E_{\beta,\beta}(-(t-s)^{\beta}|\xi|^{\alpha})\hat{v}_\xi \right|^2d\xi \\
\le&\int_{\mathbb{R}^n}(1+|\xi|^{2\sigma+2-2\delta})E_{\beta,\beta}^2(-(t-s)^{\beta}|\xi|^{\alpha})|\xi|^{2\delta}|\hat{v}_\xi|^2\, d\xi \\
\le&C(1+(t-s)^{-\frac{\beta(2\sigma+2-2\delta)}{\alpha}})\|v\|_{H^{\delta,2}}^2.
\end{aligned}
\end{eqnarray}
Since $E_{\beta,\beta}(-s)\sim C_2s^{-2}$ as $s\to\infty$, the last inequality is valid
if  $2\sigma+2-2\delta\le 4\alpha$.This computation implies that
\begin{gather}\label{eq:estH2}
\|\int_0^t T_{\alpha}^{\beta}(t-s)\nabla\cdot v\,ds\|_{H^{\sigma,2}}
\le C(T)(1+\int_0^t(t-s)^{\beta-1-\beta\frac{2\sigma+2-2\delta}{2\alpha}})\|v\|_{H^{\delta,2}}
\end{gather}
Picking $\sigma=\delta=0$ in \eqref{eq:estH1}-\eqref{eq:estH2} (with $u=\rho, v=a_n\rho^+$), and applying the method in \cite[Appendix]{llljg17}, we find that
\begin{gather}\label{eq:regularityrhon}
\varrho^{(n)}\in C^{\beta}([0,\infty), L^2(\mathbb{R}^n))\cap C^{\infty}((0,\infty), L^2(\mathbb{R}^n)).
\end{gather}
Using this time regularity, we have that in $C([0,\infty), H^{-\alpha, 2}(\mathbb{R}^n))$, $\varrho^{(n)}$ solves the following initial value problem in strong sense:
\begin{equation}\label{modifiedKS}
\left\{
  \begin{aligned}
  &^c_0D_t^\beta  \varrho^{(n)}+(-\triangle)^{\frac{\alpha}{2}}\varrho^{(n)}=-\nabla\cdot((\varrho^{(n)})^+ a_n(x,t))  && \mbox{in } (x,t)\in \mathbb{R}^n\times(0,\infty),\\
  &\varrho(x,0)=\varrho_0^{(n)}(x)\ge 0,\\
  \end{aligned}
     \right.
\end{equation}
With the time regularity and \eqref{Integrationby} in Lemma \ref{lmm:defconsistency}, we find that in $H^{-\alpha,2}(\mathbb{R}^n)$
\begin{equation}\label{eq:Capvarrho}
^c_0D_t^\beta \varrho^{(n)}=\frac{1}{\Gamma(1-\beta)}\left(\frac{\varrho^{(n)}-\varrho_0^{(n)}}{t^\beta}+\beta\int_0^t\frac{\varrho^{(n)}-\varrho^{(n)}(x,s)}{(t-s)^{\beta+1}}ds \right).
\end{equation}

It is clear that
\[
\rho \mapsto a_n\rho^+
\]
is bounded in $L^2(\mathbb{R}^n)$ and $H^{1,2}(\mathbb{R}^n)$. By interpolation, this mapping is bounded in $H^{\delta,2}, 0\le \delta\le 1$.
We now denote
\[
v^n=(\varrho^{(n)})^+a_n.
\]
We know $\rho^{n}\in C([0, T]; L^2(\mathbb{R}^n))$. Hence,
we can pick $\delta=0$, $2\sigma+2-2\delta<2\alpha$, or $\sigma<\alpha-1$ in \eqref{eq:estH1}-\eqref{eq:estH2} to find
\[
\|\rho^n\|_{H^{\sigma,2}}\le Ct^{-\frac{\sigma \beta}{\alpha}}+
C(T)(1+\int_0^t(t-s)^{\beta-1-\beta\frac{2\sigma+2-2\delta}{2\alpha}})\|v^n\|_{2}\le Ct^{-\frac{\sigma \beta}{\alpha}}.
\]
This allows us to pick $\delta\in (0, \alpha-1)$ and  $\sigma<\alpha+\delta-1$, which then improves the regularity of $\rho^n$. Clearly, we can continue this process in finite steps so that $\delta=1$ and $\sigma<\alpha$. Hence, for $t>0$,
\[
\varrho^{(n)}\in H^{\sigma, 2}(\mathbb{R}^n),~~\sigma\in [0, \alpha).
\]
 Consequently, for $t>0$, $\varrho^{(n)}\in H^{1,2}(\mathbb{R}^n)$ and $(-\Delta)^{\frac{\alpha}{2}}\varrho^{(n)}\in H^{-\epsilon,2}(\mathbb{R}^n)$ for any $\epsilon>0$.

The right hand side of \eqref{eq:Capvarrho} also makes sense in $C([0,\infty), L^2(\mathbb{R}^n))$. Since $(\varrho^{(n)})^-\in L^2(\mathbb{R}^n)$, we can multiply $(\varrho^{(n)})^-$
on both sides of \eqref{eq:Capvarrho} for $t>0$ and take integral with respect to $x$. Together with the facts $\varrho^{(n)}\in H^{1,2}(\mathbb{R}^n)$ and $(-\Delta)^{\frac{\alpha}{2}}\varrho^{(n)}\in H^{-1,2}(\mathbb{R}^n)$ for $t>0$, we can pair both sides of \eqref{modifiedKS} with $(\varrho^{(n)})^-$ with respect to $x$ to get for any $t>0$:
\begin{equation}\label{modifiedKS1}
  \int_{\mathbb{R}^n} {^c_0D_t^\beta}  \varrho^{(n)}(\varrho^{(n)})^- dx+\langle (-\triangle)^{\frac{\alpha}{2}}\varrho^{(n)}, (\varrho^{(n)})^-\rangle=-\int_{\mathbb{R}^n}\nabla\cdot((\varrho^{(n)})^+ a_n(x,t))(\varrho^{(n)})^-dx.
\end{equation}

For the term  $ \int_{\mathbb{R}^n} {^c_0D_t^\beta}  \varrho^{(n)}(\varrho^{(n)})^- dx$ (for notational convenience, we use $\varrho$ to represent a general $\varrho^{(n)}$), \eqref{Integrationby} in Lemma \ref{lmm:defconsistency} gives that
\begin{multline}\label{5.7}
\int_{\mathbb{R}^n}(^c_0D_t^\beta\varrho)\varrho^-dx=\frac{1}{\Gamma(1-\beta)}\Big(\int_{\mathbb{R}^n}\frac{-|\varrho^-(t)|^2-\rho_0(x)\varrho^-(x,t)}{t^\beta}dx \\
-\beta\int_{\mathbb{R}^n}\int_0^t\frac{|\varrho^+(x,s)\varrho^-(x,t)|}{(t-s)^{\beta+1}}dsdx
-\beta\int_{\mathbb{R}^n}\int_0^t  \frac{(\varrho^-(x,t)-\varrho^-(x,s))\varrho^-(x,t)}{(t-s)^{\beta+1}}
dsdx\Big) \\
 \le \frac{1}{\Gamma(1-\beta)}
\left(-\frac{\|\varrho^-(x,t)\|_2^2}{2t^{\beta}}
-\beta \int_0^t
\frac{(\varrho^-(x,t)-\varrho^-(x,s))\varrho^-(x,t)\,dx}{(t-s)^{\beta+1}}\,ds\right).
\end{multline}
Here we have used the nonnegativity of $\varrho_0$ and $\varrho^-$. Note that
\[
-(a-b)a\le -\frac{a^2-b^2}{2},
\]
Applying the above inequality to \eqref{5.7}, we have
\begin{equation}\label{5.8}
\int_{\mathbb{R}^n}(^c_0D_t^\beta\varrho)\varrho^-dx
\le
 \frac{1}{\Gamma(1-\beta)}
\left(-\frac{\|\varrho^-\|_2^2}{2t^{\beta}}
-\beta \int_0^t
\frac{\|\varrho^-(t)\|_2^2-\|\varrho^-(s)\|_2^2)}{2(t-s)^{\beta+1}}\,ds\right).
\end{equation}
Note that $\|\varrho_0^-\|_2^2=0$. Applying Equation \eqref{Integrationby} to $\|\varrho^-\|_2^2$ (due to the regularity results \eqref{eq:regularityrhon}), \eqref{5.8} implies that for any $n\geq1$:
\begin{equation}\label{5.9}
\int_{\mathbb{R}^n}(^c_0D_t^\beta\varrho^{(n)})(\varrho^{(n)})^-dx \le -\frac{1}{2}(^c_0D_t^\beta\|(\varrho^{(n)})^-\|_2^2).
\end{equation}

With Corollary \ref{cor:lemfrac}, we have that for any $t>0$
\begin{equation}\label{5.10}
\langle (-\Delta)^{\frac{\alpha}{2}}\varrho^{(n)}, (\varrho^{(n)})^-\rangle\le 0,
\end{equation}
since $(-\Delta)^{\frac{\alpha}{2}}\varrho^{(n)}\in H^{-1,2}(\mathbb{R}^n)$ while $\rho^{n}\in H^{1,2}$ for $t>0$.

Further, since for any $t>0$, $\varrho^{(n)}\in H^{1,2}(\mathbb{R}^n)$, we have $(\varrho^{(n)})^-\in H^{1,2}(\mathbb{R}^n)$ and $\nabla (\varrho^{(n)})^-=-1_{\varrho^{(n)}\le 0}\nabla \varrho^{(n)}$. Then, we obtain that
\begin{equation}\label{5.11}
-\langle \nabla\cdot((\varrho^{(n)})^+a_n(x,t)), (\varrho^{(n)})^-\rangle
=0.
\end{equation}

Combining \eqref{modifiedKS1},  \eqref{5.9}, \eqref{5.10} and \eqref{5.11}, for $t>0$, it holds that
\[
0\le -\frac{1}{2}(^c_0D_t^\beta\|(\varrho^{(n)})^-\|_2^2).
\]
Since $\|(\varrho^{(n)})^-\|_2^2$ is continuous in time, $(\varrho^{(n)})^-=0$ follows from Lemma \ref{lmm:volterra}.

This means that $\varrho^{(n)}$ indeed satisfies the following equation
 \begin{gather}
\varrho^{(n)}(t)=S_{\alpha}^{\beta}(t)\varrho_0^{(n)}-\int_0^T  T_{\alpha}^{\beta}(t-s)\nabla\cdot((\varrho^{(n)}) B_{\frac{1}{n}}(\rho)))\,ds,
\end{gather}
and $\varrho^{(n)}\ge 0$.

Note that $\varrho_0^{(n)}\in L^p(\mathbb{R}^n)$. Since $B_{\frac{1}{n}}(\rho)$ is smooth and bounded, we find that
$\varrho^{(n)}\in C([0,\infty), L^p(\mathbb{R}^n))$. Then, for $t\in [0, T]$,
\[
\begin{split}
&\|\varrho^{(n)}(t)-\rho(t)\|_p\\
\le &\|\varrho_0^{(n)}-\rho_0\|_p+
\int_0^t(t-s)^{-\frac{n\beta}{\alpha}(\frac{1}{q}-\frac{1}{p})
-\frac{\beta}{\alpha}+\beta-1}\|\varrho^{(n)} B_{\frac{1}{n}}(\rho)
-\rho B(\rho)\|_q\,ds\\
\le &\|\varrho_0^{(n)}-\rho_0\|_p+
\int_0^t(t-s)^{-\frac{n\beta}{\alpha}(\frac{1}{q}-\frac{1}{p})
-\frac{\beta}{\alpha}+\beta-1}\|\varrho^{(n)}-\rho\|_p(s) \|B_{\frac{1}{n}}(\rho(s))\|_{\frac{pq}{p-q}}\,ds\\
+ &\int_0^t(t-s)^{-\frac{n\beta}{\alpha}(\frac{1}{q}-\frac{1}{p})
-\frac{\beta}{\alpha}+\beta-1}\|\rho\|_p\|J_{\frac{1}{n}}*B(\rho)-B(\rho)\|_{\frac{pq}{p-q}}\,ds\\
\le & C\int_0^t(t-s)^{-\frac{n\beta}{\alpha}(\frac{1}{q}-\frac{1}{p})
-\frac{\beta}{\alpha}+\beta-1}\|\varrho^{(n)}-\rho\|_p(s)\,ds
+\delta_n(T)
\end{split}
\]
where
\[
\delta_n(T)=\|\varrho_0^{(n)}-\rho_0\|_p+\sup_{0\le t\le T}\int_0^t(t-s)^{-\frac{n\beta}{\alpha}(\frac{1}{q}-\frac{1}{p})
-\frac{\beta}{\alpha}+\beta-1}\|\rho\|_p\|J_{\frac{1}{n}}*B(\rho)-B(\rho)\|_{\frac{pq}{p-q}}\,ds
\] goes to zero as $n\to\infty$.

By the comparison principle (Lemma \ref{lmm:comparison}), we have
\[
\|\varrho^{(n)}(t)-\rho(t)\|_p \le u(t),
\]
where $u(t)$ solves the equation
\[
_0^cD_t^{-\frac{n\beta}{\alpha}(\frac{1}{q}-\frac{1}{p})
-\frac{\beta}{\alpha}+\beta}u_n(t)=Cu_n(t), ~~
u_n(0)=\delta_n(T).
\]
Clearly, as $n\to\infty$, $u_n(t)\to 0$ for $t\in [0, T]$, and then we conclude that
\[
\varrho^{(n)}(t)\to \rho(t),~~\mbox{in}~~C([0, T], L^p(\mathbb{R}^n)).
\]
 Since $\varrho^{(n)}(t)\ge 0$, we then have $\rho(t)\ge 0$ for $t\in [0, T]$. Since $T\in (0, T_b)$ is arbitrary, the claim follows.
\end{proof}

\begin{cor}
In Theorem \ref{thm:LpL1} (or Theorem \ref{thm:L1critical},Theorem \ref{thm:weighted}), besides the conditions listed there, if further we have $\rho_0\ge 0$, then $\rho(x, t)\ge 0$ and the total mass is conserved, that is
\[
M=\int_{\mathbb{R}^n}\rho\,dx=\int_{\mathbb{R}^n}\rho_0\,dx.
\]
\end{cor}

%================================================================================================================================
\section{Finite time blow up of solutions }\label{sec:blowup}
In this section, we investigate the blowup behaviors for \eqref{generalized Keller-Segel}. For the usual parabolic-elliptic Keller-Segel equations, a strategy of proof for blowup relies on the second moment method (see, for example, the celebrated work of Nagai \cite{nagai2001}).
As mentioned in Remark \ref{rmk:moment} (see also \cite[Section 2]{BK}), if $\alpha\in (1,2)$, the solution usually does not have moment of order $\nu\ge \alpha$. Hence, the standard technique using second moment does not work for $\alpha\in (1, 2)$.
If we focus on the fundamental solution $P$, a straightforward corollary of Lemma \ref{lmm:asymp} is
\begin{gather}
\int_{\mathbb{R}^n} |x|^{\nu}P\,dx\le Ct^{\frac{\beta\nu}{\alpha}},
\end{gather}
for $\nu\in (1, \alpha)$ if $\alpha\in (1,2)$ or  for $\nu\in (1,2]$ if $\alpha=2$.   Hence moment of order $\nu\in (1, \alpha)$ might work. (Of course for $\alpha=2$, one can consider second moment.)

Using the $\nu$-moment to prove the blowup with Caputo time fractional derivative follows from the similar approach as the proof of \cite[Theorem 2.3]{BK}. For the sake of completeness, we provide a detailed proof here.

Consider the following function which will be used to construct the moment:
\begin{equation}\label{FI}
\varphi(x):=(1+|x|^2)^{\frac{\nu}{2}}-1.
\end{equation}
 The following lemma has been proved in \cite{BK} to justify that $\varphi$ is equivalent to $|x|^\nu$ and has some good properties :
\begin{lem}\cite{BK}\label{lmm:testvarphi}
\begin{enumerate}[(i)]
\item For $\varepsilon>0$, there exists $C(\varepsilon)>0$ such that
\begin{equation}\label{FI1}
\varphi(x)\leq |x|^\nu\leq \varepsilon+C(\varepsilon)\varphi(x)
\end{equation}

\item For $1<\nu<\alpha<2$, or $1<\nu\le\alpha=2$,
\begin{equation}\label{FI2}
(-\triangle)^{\frac{\alpha}{2}}\varphi(x)\in L^\infty(\mathbb{R}^n).
\end{equation}

\item For $1<\nu\leq 2$, there exists $K=K(\nu)>0$ such that the following inequality
\begin{equation}\label{FI3}
\frac{|x-y|^2}{1+|x|^{2-\nu}+|y|^{2-\nu}}\leq\frac{1}{K}\big(\nabla\varphi(x)-\nabla\varphi(y)\big)\cdot(x-y),~\forall x, y\in \mathbb{R}^n.
\end{equation}
\end{enumerate}
\end{lem}

In order to study the finite time blowup of solutions to equation \eqref{generalized Keller-Segel}, we need some auxiliary results. We then have the following claim
\begin{prop}\label{pro:moments}
Let $\rho$ be the mild solutions in Theorem \ref{thm:LpL1} (or Theorem \ref{thm:L1critical}, or Theorem \ref{thm:weighted}). Let $\nu\in (1, \alpha)$ if $\alpha\in (1,2)$ and $\nu\in (1,\alpha]$ if $\alpha=2$. In addition, if $\rho_0\ge 0$ and $\rho_0\in L^1(\mathbb{R}^n, (1+|x|^2)^{\nu/2}dx)$, then in the interval of existence of the mild solutions,
\[
\omega(t):=\int_{\mathbb{R}^n}\varphi(x)\rho(x,t)\,dx=\int_{\mathbb{R}^n}((1+|x|^2)^{\frac{\nu}{2}}-1)\rho(x,t)\,dx<\infty.
\]
Further, $\omega(t)$ is continuous and satisfies the equation
\begin{eqnarray}
\begin{aligned}\label{6.1}
_0^cD_t^{\beta}\omega(t)&=-\int_{\mathbb{R}^n}\rho(x,s)(-\Delta)^{\frac{\alpha}{2}}\varphi\,dx\\
&-\frac{s_{n,\gamma}}{2}\iint_{\mathbb{R}^n\times\mathbb{R}^n}
\frac{\rho(x,s)\rho(y,s)}{|x-y|^{n-\gamma+2}}(\nabla\varphi(x)-\nabla\varphi(y))\cdot(x-y)\,dxdy.
\end{aligned}
\end{eqnarray}
\end{prop}

\begin{proof}
Below, we provide a uniform proof for all the conditions in the theorems (although the first part  of the claims is trivial for Theorem \ref{thm:weighted}. ).

Let $T_b$ be the largest time of existence. We fix $T\in (0, T_b)$ and let $\rho$ be the mild solution on $[0, T]$. Define $\rho(t)=\rho(T)$ for $t\ge T$. We consider the following regularized system
\begin{gather}\label{eq:rhonreg}
 \rho^{n}(t)=S_{\alpha}^{\beta}(t)\rho_0-\int_0^t  T_{\alpha}^{\beta}(t-s)\nabla\cdot(\rho^{n} B_{\frac{1}{n}}(\rho))\,ds,
\end{gather}
where $B_{\frac{1}{n}}(\rho)=J_{\frac{1}{n}}*B(\rho)$ with $J_{\epsilon}=\frac{1}{\epsilon^n}J(\frac{x}{\epsilon})$.
Repeat all the arguments as in the proof of Theorem \ref{thm:positivity}, we have
\[
\rho^n\ge 0, ~~\int_{\mathbb{R}^n} \rho^n(x,t)\,dx=\int_{\mathbb{R}^n}\rho_0\,dx=:M.
\]
More importantly, by taking the difference directly, we find for some $p\ge p_c$ (depending on which theorems we are considering)
\[
\rho^n\to \rho~~\mbox{in}~~C([0, T]; L^p(\mathbb{R}^n)),~~\mbox{as}~~ n\rightarrow \infty,
\]

By equation \eqref{eq:rhonreg}, we find that in $C([0, T]; H^{-\alpha,p}(\mathbb{R}^n))$, the following equality holds in strong sense:
\begin{equation}\label{6.2}
_0^cD_t^{\beta}\rho^n=-(-\Delta)^{\frac{\alpha}{2}}\rho^n-\nabla\cdot(\rho^n B_{\frac{1}{n}}(\rho)).
\end{equation}
Testing \eqref{6.2} using $\psi\in C_c^{\infty}(\mathbb{R}^n)$, we have
\begin{gather}\label{eq:weakrhon}
_0^cD_t^{\beta}\langle \rho^n, \psi\rangle=-\int_{\mathbb{R}^n} \rho^n (-\Delta)^{\frac{\alpha}{2}}\psi\,dx+\int_{\mathbb{R}^n}\rho^n B_{\frac{1}{n}}(\rho)\cdot\nabla\psi \,dx.
\end{gather}

Now, we take $n\to\infty$ in \eqref{eq:weakrhon} which is valid since $\psi\in C_c^{\infty}(\mathbb{R}^n)$ and $\rho^n\to \rho$ in $C([0, T], L^p(\mathbb{R}^n))$, and have:
\begin{gather}\label{6.6}
_0^cD_t^{\beta}\langle \rho, \psi\rangle=-\int_{\mathbb{R}^n} \rho (-\Delta)^{\frac{\alpha}{2}}\psi\,dx+\int_{\mathbb{R}^n}\rho B(\rho)\cdot\nabla\psi \,dx.
\end{gather}
Then, \eqref{6.6} gives that
\begin{equation}\label{6.7}
\langle \rho,\psi\rangle(t)=\langle \rho,\psi\rangle(0)+\frac{1}{\Gamma(\beta)}\int_0^t(t-s)^{\beta-1}(J_1(s)+J_2(s))ds,
\end{equation}
where
\[
\begin{split}
&J_1(s)=-\int_{\mathbb{R}^n}\rho(x,s)(-\Delta)^{\frac{\alpha}{2}}\psi\,dx,\\
&J_2(s)=\int_{\mathbb{R}^n}\rho B(\rho)\cdot\nabla\psi \,dx\\
&=-\frac{s_{n,\gamma}}{2}\iint_{\mathbb{R}^n\times\mathbb{R}^n}
\frac{\rho(x,s)\rho(y,s)}{|x-y|^{n-\gamma+2}}(\nabla\psi(x)-\nabla\psi(y))\cdot(x-y)\,dxdy.
\end{split}
\]

Choose $\zeta\in C_c^{\infty}(\mathbb{R}^n)$ and $0\le \zeta\le 1$, that is, for $|x|\le 1$, $\zeta=1$, for $|x|\ge 2$, $\zeta=0$ .  Now we take $\psi_m=\varphi\zeta_m,$  where $\zeta_m=\zeta(\frac{x}{m})$. The corresponding $J_1$ and $J_2$ will be denoted by $J_{1,m}$ and $J_{2,m}$.
Direct computation verifies that $\|\nabla^2\varphi\|_{\infty}<\infty$, and consequently
\begin{gather}\label{eq:boundHess}
\begin{split}
\sup_{m\ge 1}\|\nabla^2\psi_m\|_{\infty}&=\sup_{m\ge 1}\left(\|\nabla^2\varphi\|_{\infty}\|\zeta\|_{\infty}
+2\frac{1}{m}\sup_{m\le |x|\le 2 m}|\nabla\varphi|\|\nabla\zeta\|_{\infty}+\frac{1}{m^2}\|\varphi \nabla^2\zeta\|_{\infty}\right)\\
&<\infty.
\end{split}
\end{gather}
Note that for $|x|\in [m, 2m]$, $\nabla\varphi\sim (1+|x|)^{\nu-1}\le Cm^{\nu-1}$. By Corollary \ref{cor:hls},
\begin{gather}\label{eq:ineqfordoubleint}
\iint_{\mathbb{R}^n\times\mathbb{R}^n}
\frac{\rho(x,s)\rho(y,s)}{|x-y|^{n-\gamma}}\, dxdy \le C\|\rho\|_{\frac{2n}{n+\gamma}}^2(s)<\infty,
\end{gather}
where we have used the fact that $\frac{2n}{n+\gamma}\in [1, p_c]$.

Using \eqref{eq:boundHess} and \eqref{eq:ineqfordoubleint}, we find
\[
|J_{2,m}|\le C\|\rho\|_{\frac{2n}{n+\gamma}}^2(s)<\infty
\]
and similarly
\[
I_2:=-\frac{s_{n,\gamma}}{2}\iint_{\mathbb{R}^n\times\mathbb{R}^n}
\frac{\rho(x,s)\rho(y,s)}{|x-y|^{n-\gamma+2}}(\nabla\varphi(x)-\nabla\varphi(y))\cdot(x-y)\,dxdy<\infty.
\]
Consequently, we find that as $m\to\infty$, uniformly on $[0, T]$:
\[
|I_2-J_{2,m}|\le C\iint_{\mathbb{R}^n\times\mathbb{R}^n\setminus B(0,m)\times B(0, m)}\frac{\rho(x,s)\rho(y,s)}{|x-y|^{n-\gamma}}\, dxdy\to 0.
\]

It is also clear that uniformly for $s\in [0, T]$,
\[
J_{1,m}\to I_1:=-\int_{\mathbb{R}^n}\rho(x,s)(-\Delta)^{\frac{\alpha}{2}}\varphi\,ds.
\]

Replacing $\psi$ with $\psi_m$, $J_{i}$ with $J_{i,m}$ ($i=1,2$) in  \eqref{6.7} and taking $m\to \infty$, we obtain on $[0, T]$:
\[
\lim_{m\to\infty}\langle \rho,\psi_m\rangle=\omega(0)+\frac{1}{\Gamma(\beta)}\int_0^t(t-s)^{\beta-1}(I_1(s)+I_2(s))\,ds,
\]
which is continuous and finite. However, for left hand side, we apply monotone convergence theorem and find that the limit must be $\omega(t)$. This is implies that $\omega(t)$ is continuous on $[0, T]$ and the equation \eqref{6.1} is valid.
Since $T\in (0, T_b)$ is arbitrary, the claim follows.
\end{proof}

\begin{thm}\label{thm:blowup}
Assume $n\ge 2$, $0<\beta<1, 1<\alpha\le 2, 1<\gamma\leq n$. Assume $\rho_0\ge 0$ and also the conditions in Theorem \ref{thm:LpL1} (or Theorem \ref{thm:weighted}) hold to ensure the existence of mild solutions. If one of the following conditions are satisfied,
\begin{enumerate}[(i)]
\item For $\alpha=2, \gamma=n,\nu=2$, if $\rho_0\in L^1(\mathbb{R}^n, (1+|x|^{\nu})dx)$ so that
\[
\|\rho_0\|_1>\frac{2n}{s_{n,\gamma}}.
\]

\item For some $\nu\in (1, \alpha)$ when $\alpha<2$  or $\nu\in (1, 2]$ when $\alpha=2$, with $\frac{n-\gamma+2}{\nu}>1$, if $\rho_0\in L^1(\mathbb{R}^n, (1+|x|^{\nu})dx)$ and
\[
\|\rho_0\|_1>M^*,~\int_{\mathbb{R}^n} |x|^{\nu}\rho_0(x)\,dx<\delta,
\]
for certain constants $M^*(\nu,n,\alpha,\gamma)$ and $\delta(\nu,n, \alpha,\gamma)$.

\item Suppose $\alpha+\gamma<n+2$ (or $p_c>1$) and for some $\nu\in (1, \alpha)$ when $\alpha<2$  or $\nu\in (1, 2]$ when $\alpha=2$. If $\rho_0\in L^1(\mathbb{R}^n, (1+|x|^{\nu})dx)$ satisfying
\begin{equation}\label{BlowupC}
\frac{\int_{\mathbb{R}^n}|x|^\nu \rho_0(x)dx}{\int_{\mathbb{R}^n}\rho_0(x)dx}\leq \chi\left(\int_{\mathbb{R}^n}\rho_0(x)dx\right)^{\frac{\nu}{n+2-\alpha-\gamma}},
\end{equation}
where $\chi=\delta (M^*)^{-1+\frac{\nu}{\alpha+\gamma-2-n}}$ ($M^*$, $\delta$ are the constants in (ii)).
\end{enumerate}
then the mild solution of \eqref{generalized Keller-Segel} will blow up in a finite time.
\end{thm}
\begin{proof}
Recall the function
\[
\omega(t)=\int_{\mathbb{R}^n}\varphi(x)\rho(x,t)dx.
\]
Proposition \ref{pro:moments} implies that on the existence of interval $[0, T]$, we have $\omega(t)\in [0,\infty)$ and
\begin{eqnarray}
\begin{aligned}\label{estimateofM2}
^c_0D_t^\beta \omega(t)&=-\int_{\mathbb{R}^n}\Big((-\triangle)^{\frac{\alpha}{2}}\varphi(x)\Big)\rho(x,t) dx\\
&-\frac{s_{n,\gamma}}{2}\int_{\mathbb{R}^n}\int_{\mathbb{R}^n}(\nabla\varphi(x)-\nabla\varphi(y))\cdot(x-y)\frac{\rho(x,t)\rho(y,t)}{|x-y|^{n-\gamma+2}}dxdy\\
&=I_1+I_2.
\end{aligned}
\end{eqnarray}
In the existence of interval, $M:=\int_{\mathbb{R}^n}\rho(x)dx=\int_{\mathbb{R}^n}\rho_0(x)dx>0$. This implies that $\omega(t)>0$. Hence, as long as the solution does not blow up, $\omega(t)\in (0,\infty)$.

(i). In the case that $\alpha=2, \gamma=n,\nu=2$, we find that
\[
I_1+I_2=2nM-s_{n,\gamma}M^2.
\]
Hence, if $M>\frac{2n}{s_{n,\gamma}}$, $w(t)$ will be zero in finite time. This means that the solution will blow up in finite time.

(ii). Lemma \ref{lmm:testvarphi} shows that there exists constant $C_0$ such that
\[
I_1\le C_0M.
\]
Now, consider $I_2$, by Lemma \ref{lmm:testvarphi} (Equation \eqref{FI3}), we have that
\begin{eqnarray}
\begin{aligned}\label{Jestimate}
I_2&\le -\frac{1}{2}s_{n,\gamma}K \int_{\mathbb{R}^n}\int_{\mathbb{R}^n}\frac{\rho(x,t)\rho(y,t)}{|x-y|^{n-\gamma}(1+|x|^{2-\nu}+|y|^{2-\nu})}dxdy\\
&=:-\frac{1}{2}s_{n,\gamma}KJ(t).
\end{aligned}
\end{eqnarray}
Choose $p\in (1,\infty)$, $\delta\in (0,1)$ and $s\ge 0$ such that
\begin{equation}\label{Equality}
p=\frac{1}{\delta}, s=(n-\gamma)\delta, sp'+(2-\nu)\delta p'=\nu,
\end{equation}
where $p'=\frac{p}{p-1}$. Indeed, \eqref{Equality} implies that
\begin{gather}
p=\frac{n-\gamma+2}{\nu}>1,~
s=\frac{(n-\gamma)\nu}{n-\gamma+2}\ge 0.
\end{gather}
By H\"older's inequality, we then have:
\begin{eqnarray}
\begin{aligned}\label{EstiamteofM}
&M^2=\int_{\mathbb{R}^n}\int_{\mathbb{R}^n} \rho(x,t)\rho(y,t)dxdy\\
&=\int_{\mathbb{R}^n}\int_{\mathbb{R}^n}\rho(x,t)\rho(y,t)\frac{|x-y|^s}{(1+|x|^{2-\nu}+|y|^{2-\nu})^\delta}\frac{(1+|x|^{2-\nu}+|y|^{2-\nu})^\delta}{|x-y|^s}dxdy \\
&\leq J(t)^{\frac{1}{p}}\cdot(\int_{\mathbb{R}^n}\int_{\mathbb{R}^n}\rho(x,t)\rho(y,t)|x-y|^{sp'}\big(1+|x|^{2-\nu}+|y|^{2-\nu})^{\delta p'}dxdy\big)^{\frac{1}{p'}}.
\end{aligned}
\end{eqnarray}
By \eqref{Equality} and inequality \eqref{FI1} in Lemma \ref{lmm:testvarphi}, we find
\begin{equation}\label{6.14}
\begin{split}
|x-y|^{sp'}(1+|x|^{2-\nu}+|y|^{2-\nu})^{\delta p'} &\leq
C_{1,1}\max((1+|x|^{\nu}), (1+|y|^{\nu})) \\
& \le C_{1,2}(1+\varphi(x)+\varphi(y)).
\end{split}
\end{equation}
\eqref{EstiamteofM} and \eqref{6.14} imply that
\begin{equation}\label{estimateofM1}
M^2\leq C_{1,3}^{\frac{1}{p'}}J(t)^{\frac{1}{p}}(M^2+2M \omega(t))^{\frac{1}{p'}}.
\end{equation}
and with \eqref{Jestimate} and \eqref{estimateofM1}, we obtain that
\begin{equation}\label{6.16}
I_2\le -C_1\frac{M^{2p}}{(M^2+2M \omega(t))^{p-1}}.
\end{equation}
Hence, we have the following inequality for $\omega(t)$ applying \eqref{estimateofM2}, the estimate of $I_1$ and \eqref{6.16}
\begin{equation}\label{EstiamteofM3}
^c_0D_t^\beta \omega(t)\leq C_0M-C_1\frac{M^{2p}}{(M^2+2M \omega(t))^{p-1}},
\end{equation}
where $C_0=\|(-\triangle)^{\frac{\alpha}{2}} \varphi\|_\infty$.

Clearly, there are $M^*>0$ and $\delta>0$ such that whenever $M>M^*$ and $\omega(0)\le \int_{\mathbb{R}^n} |x|^{\nu}\rho_0\,dx<\delta$, we have
\begin{equation}\label{6.17}
C_2:=C_0M-C_1\frac{M^{2p}}{(M^2+2M \omega(0))^{p-1}}<0.
\end{equation}
By the fundamental theorem for fractional calculus (Lemma \ref{lmm:volterra}), we find that on the interval of existence of solution to \eqref{generalized Keller-Segel}, $\omega(t)\le \omega(0)$. Consequently, with estimates \eqref{EstiamteofM3} and \eqref{6.17}, the inequality \eqref{AKS} in Corollary \ref{cor:gronwall} implies that
\[
\omega(t)<\omega(0)+C_2\frac{1}{\Gamma(\beta)}t^{\beta}.
\]
From the above inequality, one can deduce that there exists $0<T^*<\infty$ such that $\omega(T^*)=0$, which implies finite time blowup.

(iii). In this case, we have the extra assumption $\alpha+\gamma<n+2$ and
the mass $M$ can be any positive number.  For the purpose of proof, we fix $M_0>M^*$, where $M^*$ is the number in (ii).

Recall that if $\rho$ is a solution to \eqref{generalized Keller-Segel}, then so is
\[
\rho^{\lambda}=\lambda^{\alpha+\gamma-2}\rho(\lambda x,\lambda^{\frac{\alpha}{\beta}}t).
\]
These two solutions clearly have the same blow up behavior.
Choosing $\lambda^{\alpha+\gamma-2-n}=\frac{M_0}{M}$, we can then verify that the mass of $\rho^{\lambda}$ is $M_0$.

Direct computation shows that
\begin{eqnarray}
\begin{aligned}\label{M02}
\int_{\mathbb{R}^n}|x|^\nu\rho^\lambda(x,0)dx&=\int_{\mathbb{R}^n}|x|^\nu\lambda^{\alpha+\gamma-2}\rho(\lambda x,0)dx\\
&={(\frac{M_0}{M})}^{1-\frac{\nu}{\alpha+\gamma-2-n}}\int_{\mathbb{R}^n}|x|^\nu\rho_0(x)dx
\end{aligned}
\end{eqnarray}
Hence, if
\[
{(\frac{M_0}{M})}^{1-\frac{\nu}{\alpha+\gamma-2-n}}\int_{\mathbb{R}^n}|x|^\nu\rho_0(x)dx<\delta,
\]
or
 \[
 \int_{\mathbb{R}^n}|x|^\nu\rho_0(x)dx< \chi\big(\int_{\mathbb{R}^n} \rho_0(x)dx\big)^{1-\frac{\nu}{\alpha+\gamma-2-n}},
 \]
 where $\chi=\delta M_0^{-1+\frac{\nu}{\alpha+\gamma-2-n}}$, the solution blows up in finite time. By taking $M_0\to M^*$, we verify the expression of $\chi$ in the claim.
 \end{proof}

 \begin{rem}
 It is mentioned in \cite[Remark 2.6]{BK} that the conditions in (iii) indeed implies that $\|\rho_0\|_{p_c}$ is big which is in accordance with the result in Theorem \ref{thm:globalexis}.
 \end{rem}

 \section*{Acknowledgments}

The authors thank the anonymous referee for careful reading of the manuscript and insightful suggestions.
The work of J.-G Liu is partially supported by KI-Net NSF RNMS 11-07444 and NSF DMS grant No. 1514826. L. Wang is partially supported by NSFC (Grant No. 11771352, 11631007, 11571279, 11371293) and wishes to express her gratitude to Professor Zhouping Xin and Professor Changzheng Qu for their support.

\appendix

\section{ $L^p$ estimates of $\nabla Y$}\label{app:fund}

\begin{proof}[Proof of Proposition \ref{pro:fundY} Part (2)]

 Denote $\Omega_1=\{x\in \mathbb{R}^n: |x|^\alpha> t^\beta\}$ and $\Omega_2=\{x\in \mathbb{R}^n: |x|^\alpha\leq t^\beta\}$.

\begin{equation}\label{A1}
\|\nabla Y(x,t)\|_p^p=\int_{\Omega_1}|\nabla Y(x,t)|^p d x+\int_{\Omega_2}|\nabla Y(x,t)|^p d x=I+J.
\end{equation}
Let $r=|x|$.

For $I$, by Lemma \ref{lmm:asymp} Part (1), when $1<\alpha<2$,
\begin{eqnarray}
\begin{aligned}\label{A2}
I&\leq Ct^{p(2\beta-1)}\int_{|x|^\alpha\geq t^\beta}\frac{1}{|x|^{(n+\alpha+1)p}}dx\\
&=Ct^{p(2\beta-1)} \int^{\infty}_{t^{\frac{\beta}{\alpha}} }r^{-(n+\alpha+1)p+n-1} dr\\
&=Ct^{-\frac{n\beta}{\alpha}(p-1)+p(\beta-1-\frac{\beta}{\alpha})}.
\end{aligned}
\end{eqnarray}
For $\alpha=2$, we have
\begin{eqnarray}
\begin{aligned}\label{A3}
I&\le Ct^{(-\frac{(n+1)\beta }{2}+\beta-1)p}\int_{t^{\frac{\beta}{\alpha}}}^\infty \exp(-Cr^{\frac{\alpha}{\alpha-\beta}}t^{-\frac{\beta}{\alpha-\beta}})r^{n-1}dr\\
&=C t^{-\frac{n\beta}{\alpha}(p-1)+p(\beta-1-\frac{\beta}{\alpha})},
\end{aligned}
\end{eqnarray}
here we use the substitution $s=rt^{-\frac{\beta}{\alpha}}$.

For $J$, in the case $n<2\alpha-2$, we can compute directly that
\begin{eqnarray}
\begin{aligned}\label{A4}
J&\le C\int_{|x|\le t^{\frac{\beta}{\alpha}}} |t^{\beta-1-\frac{\beta(n+2)}{\alpha}}|^p|x|^p\,dx\\
&\le Ct^{(\beta-1-\frac{\beta(n+2)}{\alpha})p} \int_0^{t^{\frac{\beta}{\alpha}}}r^pr^{n-1}\,dr\\
&=C t^{-\frac{n\beta}{\alpha}(p-1)+p(\beta-1-\frac{\beta}{\alpha})}.
\end{aligned}
\end{eqnarray}

In the case $2\alpha-2<n$, we have
\begin{eqnarray}
\begin{aligned}\label{A5}
J\le &C\int_{|x|\le t^{\frac{\beta}{\alpha}}} |x|^{p(2\alpha-n-1)}t^{-p(\beta+1)}\,dx\\
&=Ct^{-p(\beta+1)}\int_0^{t^{\frac{\beta}{\alpha}}} r^{p(2\alpha-n-1)}r^{n-1}\,dr\\
&=C t^{-\frac{n\beta}{\alpha}(p-1)+p(\beta-1-\frac{\beta}{\alpha})},
\end{aligned}
\end{eqnarray}
If $n\le 2\alpha-1$, this holds for any $p\in [1,\infty)$. If $n>2\alpha-1$, we need $p<\frac{n}{n+1-2\alpha}$ or $p<\kappa_4$ for this integral to make sense.

In the case $n=2\alpha-2=2$, that is $n=\alpha=2$, we have
\begin{gather}\label{A6}
\begin{split}
J &\le C\int_{|x|\le t^{\frac{\beta}{\alpha}}} t^{-p(\beta+1)}|x|^p(1+|\log(|x|^{\alpha}t^{-\beta})|)^p \,dx \\
& =Ct^{-p(\beta+1)}\int_0^{\frac{\beta}{\alpha}}(1+|\log(r^{\alpha}t^{-\beta})|)^p r^{n-1}\,dr\\
& =Ct^{-\frac{n\beta}{\alpha}(p-1)+p(\beta-1-\frac{\beta}{\alpha})}.
\end{split}
\end{gather}
Putting the results \eqref{A2}-\eqref{A6} into \eqref{A1}, \eqref{eq:gradYLp} in Proposition \ref{pro:fundY} is verified.

By the asymptotic behavior of $\nabla Y$ in Lemma \ref{lmm:asymp}, for $x\in\Omega_2$, when $n\leq 2\alpha-1$, $\|\nabla Y\|_{\infty}<Ct^{\beta-1-(n+1)\frac{\beta}{\alpha}}$. For $x\in\Omega_1$, the bound is trivially obtained. Hence, when $n\leq 2\alpha-1$, the estimate also holds for $p=\infty=\kappa_4$.

In the case $2\alpha-1<n$ and $p=\frac{n}{n+1-2\alpha}$, according to the calculation above, we find that
\[
I^{\frac{1}{p}}
\le C t^{-\frac{n\beta}{\alpha}(1-\frac{1}{p})+\beta-1-\frac{\beta}{\alpha}}=Ct^{-\beta-1}
\]
is always true and thus $P\chi_{x\in\Omega_1}$ is in $L^p(\mathbb{R}^n)$ and satisfies the given bound. We now focus on the part on $\Omega_2$.
\[
\begin{split}
&d(\lambda)=|\{|x|\le t^{\frac{\beta}{\alpha}}: \nabla Y>\lambda \}|
\le |\{x: |x|< (Ct^{-\beta-1}\lambda^{-1})^{\frac{1}{n-2\alpha+1}}\}|\\
&\Rightarrow \lambda d^{\frac{n-2\alpha+1}{n}}
\le Ct^{-\beta-1}.
\end{split}
\]
This yields the desired result.
\end{proof}

%=============================================================================================================================================
\section*{References}

\bibliography{fractional}
\bibliographystyle{abbrv}

\end{document}